\documentclass[reqno]{amsart}
\usepackage{amssymb}
\newcommand{\ip}[2]{(#1,#2)} 

\newtheorem{theorem}{Theorem}[section]
\newtheorem{lemma}[theorem]{Lemma}

\newtheorem{corollary}[theorem]{Corollary}
\newtheorem{proposition}[theorem]{Proposition}
\theoremstyle{definition} \newtheorem{definition}[theorem]{Definition}
\newtheorem{example}[theorem]{Example}
\newtheorem{remark}[theorem]{Remark}

\newcommand{\Mid}{\,\Big|\,}

 \newcommand{\cC}{\mathcal{C}}
\newcommand{\cA}{\mathcal{A}} 
\newcommand{\cF}{\mathcal{F}} 
 
\newcommand{\cH}{\mathcal{H}} \newcommand{\cN}{\mathcal{N}}
 \newcommand{\cP}{\mathcal{P}}
 \newcommand{\cY}{\mathcal{Y}}

\newcommand{\SO}{\mathrm{SO}} \newcommand{\SU}{\mathrm{SU}}
 
\newcommand{\mP}{\mu_{\cP}} \newcommand{\mPh}{\widehat{\mu}}
\newcommand{\supp}{\mathrm{supp}} 
 
 \newcommand{\fg}{\mathfrak{g}}
 \newcommand{\fF}{\mathfrak{F}}
\newcommand{\fk}{\mathfrak{k}}

 \newcommand{\fs}{\mathfrak{s}}

 \newcommand{\bX}{\mathbf{X}}
\newcommand{\bT}{\mathbf{T}} 
\newcommand{\bbD}{\mathbb{D}} %

 \newcommand{\cS}{\mathcal{S}}
\newcommand{\D}{\mathbb{D}} 
\newcommand{\R}{\mathbb{R}} \newcommand{\C}{\mathbb{C}}
 \newcommand{\N}{\mathbb{N}}
\newcommand{\Z}{\mathbb{Z}} 

\newcommand{\Psp}{\mathcal{P}_{\mathrm{sp}}(\bX)}
\newcommand{\pr}{\mathrm{pr}}

\newcommand{\rU}{\mathrm{U}}

\newcommand{\Ad}{\operatorname{Ad}}
 
 \newcommand{\so}{\mathfrak{so}}
\newcommand{\id}{\mathrm{id}} 
\newcommand{\Tr}{\mathrm{Tr}}
\newcommand{\innp}{\langle \,\, ,\,\, \rangle}
 \newcommand{\I}{\mathbf{1}}
\newcommand{\osc}{\mathrm{osc}}

\newcommand{\bbT}{\mathbb{T}} \newcommand{\bbH}{\mathbb{H}}
\newcommand{\lar}{\mathop{\longrightarrow}}

\begin{document}

\title[Sampling on Commutative Spaces]{Sampling in Spaces of Bandlimited Functions on Commutative Spaces}

\author{Jens Gerlach Christensen}

\thanks{The research of J. G. Christensen was partially supported by
  NSF grant DMS-0801010, and ONR grants NAVY.N0001409103, NAVY.N000140910324}
\address{Department of mathematics,
University of Maryland, College Park}
\email{jens@math.umd.edu}

\author{Gestur \'Olafsson}\thanks{The research of G. \'Olafsson was supported by NSF Grant DMS-0801010}
\address{Department of Mathematics\\
Louisiana State University\\
Baton Rouge, LA 70803, USA }
\email{olafsson@math.lsu.edu}


\maketitle

\begin{abstract} A connected homogeneous space $\bX=G/K$ is called commutative if $G$ is a connected Lie group, $K$ is a compact subgroup
and the $B^*$-algebra $L^1(\bX )^K$ of $K$-invariant integrable function on $\bX$ is commutative. In this article we introduce the space
 $L^2_\Omega (\bX)$ of $\Omega$-bandlimited function on $\bX$ by using
 the spectral decomposition of $L^2 (\bX)$. We show that those spaces
 are reproducing kernel Hilbert spaces and determine the reproducing
 kernel. We then prove sampling results for those spaces using the
 smoothness of the elements in $L^2_\Omega (\bX)$. At the end we
 discuss the example of $\R^d$, the spheres $S^d$, compact symmetric
 spaces and the Heisenberg group realized
 as the commutative space $\rU (n)\ltimes \bbH_n/\rU (n)$.
\end{abstract}
\section*{Introduction}
\noindent
Reconstruction or approximation of functions using the values of the function or a natural transformation of the function at discrete sets of points has an old and prominent history. A well known example is the reconstruction of a function using discrete set of line integrals, a fundamental tool in Computerized Tomography. Stepping into the digital age has only made this more important. But \textit{sampling theory} as independent mathematical subject originates from the fundamental article \cite{Shannon1948}. We refer to \cite{Zayed1993}, in particular the introduction, and \cite{Zayed2004} for a good places to consult about the history of the subject.

Sampling theory is a field of interest to engineers, signal analysts
and mathematicians alike. It is concerned with the reconstruction of a
function or signal from its values at a certain collection of
points. Sampling theory is concerned with many questions:
\begin{enumerate}
\item Which classes of signals can we expect to reconstruct?
\item Which conditions do the sampling points have to satisfy?
\item How are the signals reconstructed from samples?
\item Reconstruction algorithms and error analysis.
\item The speed of the reconstruction.
\end{enumerate}

The first, and arguably most famous, result is the
Whittaker-Shannon-Kotelnikov sampling theorem which states that an
audible (band-limited) signal can be reconstructed from its values at
equidistant sample points if the samples are taken at the Nyquist
rate.  If sampling takes place at a slower rate the signal cannot be
reconstructed.  With a higher sampling rate (oversampling) than the
Nyquist rate the signal can be reconstructed by many different
methods.  Some of the developed methods also apply
in the case when the samples are not equidistant (irregular sampling)
but within the Nyquist rate.

Frames \cite{Ole,Heil} are generalization of bases and are relatively new addition to mathematics but have become increasingly important for approximation theory, reconstruction in function spaces,  time frequency analysis analysis and generalizations to shift (translation) invariant spaces on topological groups and homogeneous spaces. So there is no surprise that frames have also been widely used in sampling theory. We will not go into detail here, but would like to point out the article by J. Benedetto \cite{Benedetto1991} and by J. Benedetto and W. Heller  \cite{BenedettoHeller1990} as well as the fundamental work by Feichtinger, Gr\"ochenig and their co-authors \cite{Feicht1991,Feicht1990,FeichtGroech1990,FeichtGroech1992,FeichtGroech1992a,FeichtGroech1994,FeichtPrandey2003,FuhrGroechenig2007}. Again, we refer to \cite{Zayed1993}, in particular Chapter 10, for a good overview.

The natural generalization of the spaces of bandlimited functions on $\mathbb{R}^n$ are similarly defined translation invariant spaces of functions on Lie groups and homogeneous spaces, in particular homogeneous Riemannian spaces, a subject closely related to the coorbit theory of Feichtinger and Gr\"ochenig and more general reproducing kernel Hilbert spaces related to unitary representations of Lie groups, \cite{Christensen2011,ChristOlafss2009,ChristOlafsson2011,FeichtGroech1986,FeichtGroech1989a,FeichtGroech1989b, Fuhr2005}. Here we have some natural tools from analysis at our disposal, including a natural algebra of invariant differential operators, in particular the Laplace operator. Mostly then the space of bandlimited functions are defined in terms of boundness of the eigenvalues of Laplace operator. The bounded
geometry of the space allows us then to derive the needed Sobolev and Bernstein inequality. We point out \cite{FeichtPesenson2004,FeichtPesenson2005,FuhrGroechenig2007,GellerPesenson2011,Pesenson2008} as important contributions to the subject.

This article is organized as follows. In Section \ref{S:1} we recall some standard notation for Lie groups $G$ and homogeneous spaces $\bX$. We introduce the algebra of invariant differential operators on homogeneous spaces and connect it with the algebra of invariant polynomials on $T_{x_o}(\bX)$, where $x_o$ is a fixed base-point. We then recall some basic fact about representations, and
in particular, we introduce the space of smooth and analytic vectors.
In \ref{S:2} we discuss sampling in reproducing kernel Hilbert spaces on Lie groups. The main ideas are based on the work of Feichtinger and Gr\"ochenig.

Section \ref{S:3} deals with oscillation estimates on Lie groups. The exposition is based on \cite{Christensen2011} and uses smoothness of the the functions to derive oscillation results and hence sampling theorems.

We introduce the notion of Gelfand pairs and commutative spaces in Section \ref{S:4}. Here we assume that the group is connected, which allows us to state that $\bX=G/K$ is commutative if and only if the algebra  $\mathbb{D}(\bX)$ of $G$-invariant differential operators is commutative. We review well known facts about  positive definite functions and the spherical Fourier transform on $\bX$. One of the main results in this section is a recent theorem of Ruffino \cite{Ruffino} which allows us to identify the parameter set for the spherical Fourier transform with a subset of $\mathbb{C}^s$, where $s$ is the number of generators for $\mathbb{D}(\bX)$. The result by Ruffino generalizes statements about
Gelfand pairs related to the Heisenberg group by Benson, Jenkins,
Ratcliff and Worku \cite{Benson1996}.
In Section \ref{S:6} we relate the positive definite spherical functions to $K$-spherical representations of $G$ and introduce the vector valued Fourier transform on $\bX$. This relates the representation theory of $G$ to the harmonic analysis on $\bX$.

In Section \ref{S:7} we finally introduce the space of bandlimited functions on $\bX$. The definition is based on the support of the Plancherel-measure on $\bX$ and not the spectrum of the Laplacian on $\bX$ as in \cite{FeichtPesenson2004,FeichtPesenson2005,Pesenson2008}. We do not prove it, but in all examples the definitions of a bandlimited functions are equivalent,
but our definition allows for more general spectral sets.
Another benefit of our approach is that
one does not have to worry about injectivity radius for
the exponential function nor about
the construction of smooth partitions of unity (characteristic
functions for a disjoint cover can be used just as well).
Our sampling result is proved in \ref{S:8} using a Bernstein inequality for the space of bandlimited functions. Finally, in Section \ref{se-Ex}, we give some examples of commutative spaces and their spherical harmonic analysis. Those examples include the spheres, and more generally, compact symmetric spaces, and the Heisenberg group as a homogenous space for the group $\mathrm{U}(n)\ltimes \mathbb{H}_n$. This article is partially based on \cite{ChrisMayeliOlafss2011}.

\section{Notation and preliminaries}\label{S:1}

\subsection{Locally compact groups}
In the following $G$ denotes a locally compact group with left
invariant Haar measure $\mu_G$.  Sometimes we write $dx$ instead of
$d\mu_G(x)$.  For $1\leq p<\infty$ we let $L^p(G)$ denote the space of
equivalence classes of $p$-integrable functions on $G$ with norm
\begin{equation*}
  \| f\|_{L^p} = \Big( \int |f(x)|^p\,dx \Big)^{1/p}.
\end{equation*}
Further, let $L^\infty(G)$ denote the space of
essentially bounded function on $G$ with norm
\begin{equation*}
  \| f\|_\infty = \mathrm{ess}\,\sup_{x\in G} |f(x)|
\end{equation*}
The spaces $L^p(G)$ are Banach spaces for $1\le p\le \infty$ and
$L^2(G)$ is a Hilbert space with inner product
\begin{equation*}
  \ip{f}{g} = \int f(x)\overline{g(x)}\,dx.
\end{equation*}
When it makes sense (either the integrand is integrable or as a vector
valued integral) we define the convolution
\begin{equation*}
  f*g(x) = \int f(y)g(y^{-1}x)\,dy.
\end{equation*}
Equipped with convolution the space $L^1(G)$ becomes a
Banach algebra.  For functions on $G$ we denote the left and right
translations by
\begin{equation*}
  \ell(a) f(x) = f(a^{-1}x)
  \qquad\text{and}\qquad
  \rho(a) f(x) = f(xa)
\end{equation*}
respectively.  Now, let $K$ be a compact subgroup of $G$ with
bi-invariant Haar measure $\mu_K$.  We always normalize $\mu_K$
so that $\mu_K (K)=1$.  The same convention
applies to other compact groups and compact spaces.

If $\cA$ is a set of functions on $G$ we denote the left $K$-fixed
subset as
\begin{equation*}
  \cA^K = \{ f\in \cA \mid \ell(k)f=f   \}
\end{equation*}
and similarly the right $K$-fixed subset is denoted
\begin{equation*}
  \cA^{\rho(K)} = \{ f\in \cA \mid \rho(k) f=f   \}
\end{equation*}
Let $\bX =G/K$, $x_0=eK$ and let $\kappa:G\to \bX$ be the canonical
map $g\mapsto gx_0$.  We will identify functions on $\bX$ by
$K$-invariant functions via $f\leftrightarrow f\circ \kappa$.  The
space $\bX$ possesses a $G$-invariant measure $\mu_\bX$ and, since $K$
is compact, the $L^p$-spaces
$$
L^p(\bX) = \{ f \mid f \text{ is $\mu_\bX$-measurable and } \int
|f(x)|^p\,d\mu_\bX < \infty \}.
$$
The above map $f\mapsto f\circ \kappa$ is an isometric isomorphism $L^p(\bX)\simeq L^p(G)^{\rho (G)}$.
In particular, $L^p(\bX)$ is a
closed $G$-invariant subspace of $L^p(G)$. The projection $L^p(G)\to
L^p(\bX)$ is
\begin{equation*}
  p_K(f)(x) = \int_K f(xk)\,d\mu_K(k)\, .
\end{equation*}

If $f\in L^1(G)$ and $g\in L^p (\bX)$, $1\le p\le \infty$, then
$f*g\in L^p (\bX)$ and $\|f*g\|_p\le \|f\|_1\|g\|_p$.

If $f$ is further assumed to be left $K$-invariant, then
\begin{eqnarray*} f*g(ky)&=&\int_G f(x)g(x^{-1}ky)\, d\mu_G(x)\\
  &=&\int_G f(kx) g(x^{-1}y)\, d\mu_G(x)\\ &=&\int_G f(x)g (x^{-1}y)
  \, d\mu_G (x)\\ &=& f*g (y)\, .
\end{eqnarray*}
Thus $f*g$ is also left $K$-invariant. Denote by $m_G$ the
modular function on
$G$. Note that $m_G$ is usually denoted by $\Delta$ or $\Delta_G$ but we will need that notation
for the Laplace operator on $\bX$ respectively  $G$. We have
$m_G|_K=1$ as $K$ is compact.  Hence $m_G$ is
$K$-biinvariant. It follows that $L^1 (\bX)^K$ is invariant under the
anti-involutions $ f^\vee (x)=m_G(x)^{-1}f(x^{-1})$ and
$f^*=\overline{f^\vee}$. In particular, $L^1(\bX )^K$ is a closed
Banach $*$-subalgebra of $L^1(G)$.

\subsection{Lie theory}
Let $G$ be a connected Lie group and $K$ a compact subgroup.  Most of
the statements holds for nonconnected groups, but some technical
problems turn up as we start to deal with the Lie algebra and
invariant differential operators. We will therefore for simplicity assume $G$ commutative from the beginning.

Denote by $\fg$ the Lie algebra of $G$ and $\fk$ the Lie algebra of
$K$. Fix a $K$-invariant inner product $\innp$ on $\fg$. That is
always possible as $K$ is compact. Let $\fs:= \fk^\perp$.  Then $\fs$
is $K$-invariant and $\fs\simeq T_{x_o}(\bX)$ (as a $K$-module) via
the map
\[X\mapsto D(X)\, ,\quad D(X)
(f)(x_o):=\left.\dfrac{d}{dt}\right|_{t=0}f(\exp (tX) x_o)\, .\]
Denote also by $\innp$ the restriction of $\innp$ to $\fs\times
\fs$. As the tangent bundle on $T(\bX)$ is isomorphic to $G\times_K \fs$ as a $G$-bundle it
follows that the restriction of $\innp$ to $\fs$ defines a
$G$-invariant Riemannian structure on $\bX$.

Let $D : C_c^\infty (\bX) \to C^\infty_c(\bX)$ be a differential
operator. For $g\in G$ let $g\cdot D :C^\infty_c(\bX) \to
C^\infty_c(\bX)$ be the differential operator
\[g\cdot D (f)(x) = D(\ell({g^{-1}})f)(g^{-1} x)\, .\] $D$ is said to
be $G$-invariant if $g\cdot D=D$ for all $g\in G$. Thus $D$ is
$G$-invariant if and only if $D$ commutes with left translation,
$D(\ell(g) f)=\ell(g) D(f)$. Denote by $\bbD (\bX)$ the algebra of
$G$-invariant differential operators on $\bX$. The algebra $\bbD
(\bX)$ has a simple description. For a polynomial function $p : \fg\to
\C$ define a left-invariant differential operator $D_p : C_c^\infty
(G)\to C_c^\infty (G)$ by
\begin{eqnarray} D_p(f)(g)&:=&p\left(\dfrac{\partial} {\partial
      t_1},\ldots ,\dfrac{\partial} {\partial t_m}\right)f(g\exp
  (t_1X_1+\ldots + t_mX_m))|_{t_1=\ldots =t=0}\label{Dp}\\
  &=&p\left(\dfrac{\partial} {\partial t_1},\ldots ,\dfrac{\partial}
    {\partial t_m}\right)f(g\exp (t_1X_1)\cdots \exp(
  t_mX_m))|_{t_1=\ldots =t=0}\label{Dp1}
\end{eqnarray}
where we have extended our basis of $\fs$ to a basis $X_1,\ldots ,X_m$
of $\fg$. If $p$ is a $K$-invariant then
\begin{eqnarray*}
  D_p(f)(gk)&=&p\left(\dfrac{\partial} {\partial t_1},\ldots
    ,\dfrac{\partial} {\partial t_m}\right)f(gk\exp (t_1X_1+\ldots +
  t_nX_m))|_{t_1=\ldots =t_m=0}\\
  &=&p\left(\dfrac{\partial} {\partial t_1},\ldots
    ,\dfrac{\partial} {\partial t_n}\right)f(g\exp (t_1X_1+\ldots +
  t_mX_m)k)|_{t_1=\ldots =t_m=0}
\end{eqnarray*}
for all $k\in K$. Hence, if $p$ is $K$-invariant and $f$ is right
$K$-invariant it is clear from (\ref{Dp1}) that $D_p$ only depends on
the polynomial $q=p|_{\fs}$ and $D_pf=D_qf$ is right $K$-invariant and
defines a function on $\bX$. Hence $D_q$ is a $G$-invariant
differential operator on $\bX$.

Denote by $S (\fs)$ the symmetric algebra over $\fs$. Then $S(\fs )$ is commutative and isomorphic to the algebra of polynomial functions.

\begin{theorem} The map $S(\fs)^K\to \bbD (\bX)$ is bijective.
\end{theorem}
\begin{proof} This is Theorem 10 in \cite{He63}.
\end{proof}

\begin{remark} If we take $p(X)=\|X\|^2$, then $D_p=:\Delta$ is the Laplace operator on $\bX$.
\end{remark}

\begin{remark} The algebra $\bbD (\bX)$ is not commutative in general. Hence the above map is not necessarily an algebra
homomorphism.
\end{remark}

For a fixed basis $X_1,\ldots ,X_m$ for $\fg$ it will ease our
notation to introduce the differential operator $D^\alpha:C_c^\infty
(G) \to C_c^\infty(G)$ for a multi-index $\alpha$ of length $k$ with
entries between $1$ and $m$:
\begin{equation*}
  D^\alpha f(x)
  = D(X_{\alpha(k)}) D(X_{\alpha(k-1)})\cdots D(X_{\alpha(1)})f(x).
\end{equation*}

\subsection{Representation theory}
Let $\pi$ be a representation of the Lie group $G$ on a Hilbert space
$\cH$.  Then $u\in \cH$ is called \textit{smooth} respectively
\textit{analytic} if the $\cH$-valued function $\pi_u(x) = \pi (x)u$
is smooth respectively analytic. Denote by $\cH^\infty$, respectively
$\cH^\omega$, the space of smooth, respectively analytic, vectors in
$\cH$.  For $u\in \cH^\infty$ and $X\in \fg$ let
\[\pi^\infty (X)u:=\lim_{t\to 0}\frac{\pi (\exp tX)u-u}{t}\] and
$\pi^\omega (X):=\pi^\infty (X)|_{\cH^\omega}$. We have
\begin{lemma} Let $(\pi,\cH)$ be a unitary representation of $G$. Then
  the following holds:
  \begin{enumerate}
  \item The space $\cH^\infty$ is $G$-invariant.
  \item $\pi^\infty (\fg )\cH^\infty \subseteq \cH^\infty$ and
    $(\pi^\infty ,\cH^\infty)$ is a representation of $\fg$. In particular
    \[\pi^\infty ([X,Y])=\pi^\infty( X)\pi^\infty (Y)-
    \pi^\infty (Y)\pi^\infty (X)\, \]
  \item $\pi^\infty (\Ad (g) X)=\pi (g) \pi^\infty (X)\pi (g^{-1})$,
  \item $\pi^\infty (X)^*|_{\cH^\infty}=-\pi^\infty (X)$.
  \item $\cH^\infty$ is dense in $\cH$.
  \end{enumerate}
\end{lemma}

Corresponding statements are also true for $\cH^\omega$.  To show that
$\cH^\infty$ is dense in $\cH$ let $f\in L^1(G)$.  Define $\pi (f):
\cH \to \cH$ by
\[\pi (f)u=\int_G f(x)\pi (x)u\, d\mu_G (x)\, .\] Then $\|\pi (f)\|\le
\|f\|_1$, $\pi (f*g)=\pi (f)\pi (g)$ and $\pi (f^*)=\pi (f)^*$. Thus,
$\pi : L^1 (G)\to B(\cH )$ is a continuous $*$-homomorphism. If $f\in
C_c^\infty (G)$ then it is easy to see that $\pi(f)\cH \subseteq
\cH^\infty$. The main step in the proof is to show that
\[\pi^\infty (X)\pi (f)u=\pi (\ell^\infty(X) f)u\] where
\[\ell^\infty(X) f(x)=\lim_{t\to 0}\frac{f(\exp (-tX)x)-f(x)}{t}\, .\]

\begin{lemma}\label{le-pifHK} If $\{U_j\}$ is a decreasing sequence of
  $e$-neighborhoods such that $\bigcap U_j=\{e\}$ and $f_j\in
  C_c^\infty (G)$ is so that $f_j\ge 0$, $\supp f_j\subset U_j$, and
  $\|f\|_1=1$, then $\pi (f_j)u \to u$ for all $u\in \cH$. In
  particular, $\cH^\infty$ is dense in $\cH$.
\end{lemma}

\section{Reconstruction in reproducing kernel Hilbert spaces}\label{S:2}
\noindent
We will be concerned with sampling in subspaces of the Hilbert space
$L^2(G)$.
We start with the definition of a frame due to Duffin and Schaeffer
\cite{Duffin1952}. For further references see the introduction.
\begin{definition}
  For a Hilbert space $\cH$ a set of vectors $\{ \phi_i\}\subseteq
  \cH$ is a frame for $\cH$ if there are constants $0<A\leq B<\infty$
  such that
  \begin{equation*}
    A \| f\|_{\cH}^2 \leq \sum_{i} |\ip{f}{\phi_i}|^2 \leq B \| f\|_H^2
  \end{equation*}
  for all $f\in \cH$.
\end{definition}
These conditions ensure that the
frame operator $S:\cH\to \cH$ given by
\begin{equation*}
  Sf = \sum_i \ip{f}{\phi_i}\phi_i
\end{equation*}
is invertible and that $f$ can be reconstructed by
\begin{equation*}
  f = \sum_i \ip{f}{\phi_i}  \psi_i .
\end{equation*}
where $\psi_i=S^{-1}\phi_i$. The sequence $\{\psi_i\}$ is also a frame
called the dual frame. In general there are other ways to reconstruct
$f$ form the sequence $\{(f,\phi_i)\}$.

The inversion of $S$ can be carried out via the Neumann series
\begin{equation} \label{eq-neumann-series} S^{-1} = \frac{2}{A+B}
  \sum_{n=0}^\infty \Big(I-\frac{2}{A+B}S\Big)^n
\end{equation}
which has rate of convergence $\| I-\frac{2}{A+B}S\|\leq
\frac{B-A}{A+B}$ (which is the best possible for optimal frame bounds
\cite{Li1993,Grochenig1993a,Lim1998}).

One way to reconstruct $f$ from samples is to assume that the point
evaluations $f\mapsto f(x)$ is continuous and hence given by an inner
product $f(x)=(f,g_x)$, $g_x\in \cH$, and that there exists a sequence
$\{x_j\}$ in $G$ such that $\{g_{x_j}\}$ is a frame. A reasonable
class of functions to work with are therefore reproducing kernel
Hilbert spaces. A classical reference for reproducing kernel
Hilbert spaces is \cite{Aronszajn1950}. A Hilbert space $\cH$ of
functions on $G$ is called a reproducing kernel Hilbert space if point
evaluation is continuous, i.e. if for every $x\in G$ there is a
constant $C_x$ such that for all $f\in \cH$
\begin{equation*}
  |f(x)| \leq C_x \| f\|_{\cH}.
\end{equation*}

Here are the main facts about closed reproducing kernel subspaces of
$L^2(G)$ which is all what we will need here:
\begin{proposition}\label{pr-repKerHsp}
  If $\cH$ is a closed and left invariant reproducing kernel subspace
  of $L^2(G)$ then
  \begin{enumerate}
  \item There is a $\phi\in\cH$ such that $f=f*\phi$ for all $f\in
    \cH$.
  \item The functions in $\cH$ are continuous.
  \item The kernel $\phi$ satisfies $\overline{\phi(x^{-1})} =
    \phi(x)$ so $f(x)= f*\phi (x)=(f,\ell({x})\phi).$
  \item The mapping $f\mapsto f*\phi$ is a continuous projection from
    $L^2(G)$ onto $\cH$. In particular $\cH =\{f\in L^2(G)\mid f*\phi
    =f\}$.
  \end{enumerate}
\end{proposition}

\begin{proof} Here are the main ideas of the proof.  By Riesz'
  representation theorem there is a $g_x\in H$ such that
  \begin{equation*}
    f(x) = \int f(y)\overline{g_x(y)}\,dy\, .
  \end{equation*}
  Let $g (x):= g_e(x)$.  The left invariance of $\cH$ ensures that
  \begin{equation*}
    f(x)=[\ell (x^{-1})f](e)    = \int f(xy)\overline{g_e(y)}\,dy
    = \int f(y)\overline{g (x^{-1}y)}\,dy =(f,\ell (x)g)\, .
  \end{equation*}
  Hence $g_x(y)=g(x^{-1}y)$. We also have
  \[g(x^{-1}y)=(g_x,g_y)=\overline{(g_y,g_x)}=\overline{g(y^{-1}x)}\,
  .\]
  Thus, if we set $\phi(x)=\overline{g (x^{-1})}$, which agrees
  with $g^*$ in case $G$ is unimodular, we get $f =f*\phi$, which in
  particular implies that $\cH\subseteq C (G)$ as claimed.

  Assume that $f*\phi =f$ and that $f\perp \cH$. Then $f(x)=(f,g_x)=0$
  as $g_x\in\cH$. Hence $f=0$ and $\cH =L^2(G)*\phi =\{f\in L^2(G)\mid
  f*\phi = f\}$.
\end{proof}
\begin{remark}
  It should be noted that several functions $\phi\in L^2(G)$ could
  satisfy $f=f*\phi$ for $f\in \cH$. Just take an arbitrary function
  $\eta$ such that $\overline{\eta^\vee}\in \cH^\perp$. Then $f*(\phi
  +\eta)=f*\phi $. The restriction that $\phi\in \cH$ ensures
  uniqueness of $\phi$.  Example could be sinc functions for spaces of
  larger bandwidth than $\cH$.
\end{remark}

The sampling theory of Feichtinger, Gr\"ochenig and F\"uhr, see the introduction for referecnes,
builds
on estimation of the variation of a function under small right
translations.  The local oscillations were therefore introduced as
follows: For a compact neighbourhood $U$ of the identity define
\begin{equation*}
  \osc_U(f) = \sup_{u\in U} |f(x)-f(xu^{-1})|
\end{equation*}

Before stating the next result we need to introduce a reasonable
collection of points at which to sample: For a compact neighbourhood
$U$ of the identity, the points $x_i$ are called $U$-relatively
separated if the $x_iU$ cover $G$ and there is an $N$ such that each
$x\in G$ belongs to at most $N$ of the $x_iU$'s.

\begin{lemma}
  Let $\cH$ be a reproducing kernel subspace $L^2(G)$ with reproducing
  convolution kernel $\phi$.  Assume that for any compact
  neighbourhood $U$ of the identity there is a constant $C_U$ such
  that for any $f\in \cH$ the estimate $\| \osc_U(f)\|_{L^2} \leq C_U
  \| f\|_{\cH}$ holds.  If we can choose $U$ such that $C_U <1$, then
  for any $U$-relatively separated points $\{ x_i\}$ the norms $\| \{
  f(x_i)\}\|_{\ell^2}$ and $\| f\|_{L^2}$ are equivalent, and
  $\ell({x_i})\phi$ forms a frame for $\cH$.
\end{lemma}

\begin{proof}
  \begin{align*}
    \| \{f(x_i) \}\|^2_{\ell^2}
    &= |U|^{-1} \Big\| \sum_i |f(x_i)|^2 \I_{x_iU} \Big\|_{L^1} \\
    &\leq |U|^{-1} \Big\| \sum_i |f(x_i)|\I_{x_iU} \Big\|^2_{L^2} \\
    &\leq |U|^{-1} \Big(\Big\| \sum_i |f(x_i)-f|\I_{x_iU} \Big\|_{L^2}
    + \Big\| \sum_i |f|\I_{x_iU} \Big\|_{L^2}\Big)^2 \\
    &\leq |U|^{-1} \Big( \Big\| \sum_i |\osc_U (f)|\I_{x_iU}
    \Big\|_{L^2}
    + \Big\| \sum_i |f|\I_{x_iU} \Big\|_{L^2}\Big)^2 \\
    &\leq |U|^{-1} N^2 ( \|\osc_U(f) \|_{L^2} + \|f \|_{L^2} )^2 \\
    &\leq |U|^{-1} N^2 (1+C_U)^2 \| f\|^2_{L^2}.
  \end{align*}
  Here $N$ is the maximal number of overlaps between the $x_iU$'s.  To
  get the other inequality we let $\psi_i$ be a bounded partition of
  unity such that $0\leq \psi_i\leq \I_{x_i}U$ and $\sum_i \psi_i =1$.
  Then,
  \begin{align*}
    \| f\|_{L^2} &\leq \Big\| f - \sum_i f(x_i)\psi_i \Big\|_{L^2}
    + \Big\| \sum_i f(x_i)\psi_i \Big\|_{L^2} \\
    &\leq \Big\| \sum_i \osc_U(f) \psi_i \Big\|_{L^2}
    + \Big\| \sum_i |f(x_i)|\I_{x_iU} \Big\|_{L^2} \\
    &\leq \| \osc_U(f) \|_{L^2}
    + N|U| \| f(x_i) \|_{\ell^2} \\
    &\leq C_U \| f \|_{L^2}
    + N |U| \| f(x_i) \|_{\ell^2} \\
  \end{align*}
  If $C_U < 1$ then we get
  \begin{equation*}
    (1-C_U) \| f\|_{L^2}
    \leq N |U| \| f(x_i) \|_{\ell^2}.
  \end{equation*}
  This concludes the proof.
\end{proof}

\begin{remark}
  From the proof of the lemma follows that the the norm equivalence
  becomes
  \begin{equation} \label{norm-eq-samples} \left( \frac{1-C_U}{|U| N}
    \right)^2 \| f\|_{L^2}^2 \leq \| \{ f(x_i)\} \|_{\ell^2}^2 \leq
    \left( N \frac{1+C_U}{|U|} \right)^2 \| f\|_{L^2}^2,
  \end{equation}
  and thus the frame constants $A$ and $B$ can be chosen to be
  \begin{equation*}
    A = \left( \frac{1-C_U}{|U| N} \right)^2
    \qquad\text{and}\qquad
    B= \left( N \frac{1+C_U}{|U|} \right)^2.
  \end{equation*}
  It follows that the rate of convergence for the Neumann series
  \eqref{eq-neumann-series} can be estimated by
  \begin{equation*}
    \frac{B-A}{B+A}
    = \frac{N^2(1+C_U)^2 - (1-C_U)^2/N^2
    }{N^2(1+C_U)^2 + (1-C_U)^2/N^2}
    \to \frac{N^4 - 1}{N^4 + 1} \quad{\text{as $C_U\to 0$}}.
  \end{equation*}
  This shows that as the sampling points $x_i$ are chosen closer ($U$
  gets smaller) the rate of convergence can be very slow (assuming
  that we can choose the overlaps of the $x_iU$'s bounded by a certain
  $N$ even if $U$ gets smaller).
  We therefore have very little control of the rate of
  convergence in this case.
\end{remark}

To obtain operators with faster decaying Neumann series, Feichtinger
and Gr\"ochenig introduced new sampling operators.  An example is the
sampling operator $T:\cH \to \cH$ defined as
\begin{equation*}
  Tf = \sum_i f(x_i)\psi_i*\phi.
\end{equation*}
Using oscillations it is possible to estimate the norm of $I-T$ by
$C_U$:
\begin{equation*}
  \| f-Tf \|_{L^2}
  = \Big\| \Big( \sum_{i} |f-f(x_i)| \psi_i \Big)*\phi \Big\|_{L^2}
  \leq \| \osc_U f\|_{L^2}
  \leq C_U \| f\|_{L^2}
\end{equation*}
Thus $T$ is invertible on $L^2_\phi$ if $C_U <1$ with rate of
convergence of Neumann series governed directly by $C_U$.  By
increasing the rate of sampling (decreasing $U$ and thereby $C_U$)
fewer iterations are necessary in order to obtain good
approximation. This was not the case for the frame inversion above.

\section{Oscillation estimates on Lie groups}\label{S:3}
\noindent
In this section we will show how oscillation estimates can be obtained
for functions on Lie groups.

First we set up the notation. As before we let $G$ be a Lie group with
Lie algebra $\mathfrak{g}$. Fix a basis $\{
X_i\}_{i=1}^{\mathrm{dim}(G)}$ for $\mathfrak{g}$.  Denote by
$U_\epsilon$ the set
\begin{equation*}
  U_\epsilon := \left\{ 
    \exp(t_1X_1)\cdots\exp(t_nX_n)
    \Mid -\epsilon \leq t_k \leq \epsilon ,1\leq k\leq n
  \right\}.
\end{equation*}

\begin{remark} Note that $U_\epsilon$ depends on the choice of basis as well as the ordering of the
vectors. It would therefore be more natural to use sets of the form
$V_\epsilon:=\exp \{X\in\fg\mid \|X\|\leq \epsilon\}$ or even
$W_\epsilon:=\exp \{X\in\fs\mid \|X\|\leq \epsilon\}\exp \{X\in\fk\mid \|X\|\leq \epsilon\}$. Both of those sets are invariant under
conjugation by elements in $K$. The reason to use $U_\epsilon$ as defined above is, that this is the definition that works best for
the proofs! But it should be noted that $V_\epsilon, W_\epsilon \subseteq U_\epsilon$. Hence the local oscillation using either $V_\epsilon$
or $W_\epsilon $ is controlled by the local oscillation using $U_\epsilon$.
\end{remark}

Set
$$\osc_\epsilon (f)=\osc_{U_\epsilon}(f).$$
By $\delta$ we denote an $n$-tuple $\delta =
(\delta_1,\dots,\delta_n)$ with $\delta_i \in \{0,1 \}$. The length
$|\delta|$ of $\delta$ is the number of non-zero entries $| \delta | =
\delta_1+\dots +\delta_n$.  Further, define the function
$\tau_\delta:(-\epsilon,\epsilon)^n\to G$ by

\begin{equation*}
  \tau_\delta (t_1,\dots,t_n)
  = \exp({\delta_1 t_nX_1}) \cdots \exp({\delta_n t_1X_n}).
\end{equation*}

\begin{lemma}\label{lem:5}
  If $f$ is right differentiable of order $n=\mathrm{dim}(G)$ then
  there is a constant $C_\epsilon $ such that
  \begin{equation*}
    \osc_\epsilon(f)(x) \leq C_\epsilon
    \sum_{1\leq |\alpha|\leq n}
    \sum_{|\delta|=|\alpha|}
    \underbrace{\int_{-\epsilon}^\epsilon
      \cdots \int_{-\epsilon}^\epsilon}_{\text{$|\delta|$ integrals}}
    |D^\alpha f(x\tau_\delta(t_1,\dots,t_n)^{-1})|
    (dt_1)^{\delta_1} \cdots (dt_n)^{\delta_n}.
  \end{equation*}
  For $\epsilon' \leq \epsilon$ we have $C_{\epsilon'} \leq
  C_\epsilon$.
\end{lemma}

\begin{proof}
  We refer to \cite{Christensen2011} for a full proof. Instead we
  restrict ourselves to a proof in 2 dimension that easily carries
  over to arbitrary dimensions. We will sometimes write $e^X$ instead of $\exp X$.

  For $y\in U_\epsilon$ there are $s_1,s_2\in [-\epsilon,\epsilon]$ such that $y^{-1} = e^{-s_2X_2} e^{-s_1X_1}$.
  Hence
  \begin{align}
    |f(x)-f(xy^{-1})|
    &= |f(x)-f(xe^{-s_2X_2} e^{-s_1X_1})| \notag\\
    &\leq |f(x)-f(xe^{-s_2X_2})|
    + |f(xe^{-s_2X_2})-f(xe^{-s_2X_2} e^{-s_1X_1})| \notag\\
    &=
    \left| \int_{0}^{s_2} \frac{d}{dt_2}
      f(x e^{-t_2X_2}) \,dt_2  \right|
    + \left| \int_{0}^{s_1} \frac{d}{dt_1}
      f(x e^{-s_2X_2}e^{-t_1X_1}) \,dt_1  \right| \notag\\
    &\leq \int_{-\epsilon}^{\epsilon} \left|
      D(X_2)f(x e^{-t_2X_2}) \right| \,dt_2   + \int_{-\epsilon}^{\epsilon} \left|
      D(X_1)f(x e^{-s_2X_2}e^{-t_1X_1}) \right| \,dt_1.\label{eq:part1}
  \end{align}
  Since
  $$
  e^{-s_2X_2} e^{-t_1X_1}
  = e^{-t_1X_1}e^{-s_2X(t_1)}\quad \text{ with } \quad X(t)=\mathrm{Ad}(\exp (t X_1))X_2
  $$
  the term
  $|D(X_1)f(x e^{-s_2 X_2}e^{-t_1X_1})|$ can be estimated by
  \begin{align}
    | D(X_1) &f(x e^{- s_2X_2} e^{-t_1X_1})| \notag\\
    &= | D(X_1) f(x e^{-t_1X_1} e^{-s_2 X(t_1)})| \notag\\
    &\leq | D(X_1) f(x e^{- t_1X_1}e^{-s_2X (t_1)})
    - D(X_1) f(x e^{-t_1X_1} )| +  |D(X_1) f(x e^{-t_1X_1} )| \notag\\
    &=
    \left| \int_{0}^{s_1}
    \frac{d}{dt_2} D(X_1) f(x e^{-t_1X_1} e^{-t_2 X (t_1)}) \,dt_2 \right|
    +  |D(X_1) f(x e^{-t_1X_1})| \notag\\
    &=
    \left| \int_{0}^{s_1}
    D(X(t_1))D (X_1)
    f(x e^{- t_1X_1} e^{- t_2 X(t_1)}) \,dt_2 \right|
    +  |X_1 f(x e^{- t_1X_1})| \notag\\
    &\leq
    C_{\epsilon}
    \int_{-\epsilon}^{\epsilon}
    \left|D(X_2) D(X_1) f(x e^{-t_2 X_2} e^{-t_1X_1})\right|
     + \left|D(X_1) D(X_1) f(x e^{-t_2X_2} e^{-t_1X_1})\right|
    \,dt_2 \notag\\
    &\qquad
    +  |D(X_1) f(x e^{- t_1X_1})|. \label{eq:part2}
  \end{align}
  The last inequality follows since $D(X (t_2))= a(t_1)D(X_1)+b(t_1)D(X_2)$
  is a differential operator with coefficients $a$ and $b$ depending continuously, in fact analytically, on all variables.
Together \eqref{eq:part1} and \eqref{eq:part2} provide the desired
  estimate.
\end{proof}

Since right translation is continuous on $L^2(G)$ and $\sup_{u\in U}
\| r_u f\|_{L^2} \leq C_U \| f\|_{L^2}$ for compact $U$ \cite[Theorem
3.29]{Rudin1991} gives
\begin{align*}
  \| \osc_\epsilon(f)\|_{L^2} &\leq \sum_{1\leq |\alpha|\leq n}
  \sum_{|\delta|=|\alpha|} \underbrace{\int_{-\epsilon}^\epsilon
    \cdots \int_{-\epsilon}^\epsilon}_{\text{$|\delta|$ integrals}} \|
  r_{\tau_\delta(t_1,\dots,t_n)^{-1}} D^\alpha f\|_{L^2}
  (dt_1)^{\delta_1} \cdots (dt_n)^{\delta_n} \\
  &\leq C_{U_\epsilon} \sum_{1\leq |\alpha|\leq n}
  \sum_{|\delta|=|\alpha|} \underbrace{\int_{-\epsilon}^\epsilon
    \cdots \int_{-\epsilon}^\epsilon}_{\text{$|\delta|$ integrals}} \|
  D^\alpha f\|_{L^2}
  (dt_1)^{\delta_1} \cdots (dt_n)^{\delta_n} \\
  &\leq C_{U_\epsilon} \sum_{1\leq |\alpha|\leq n} {n \choose
    |\alpha|} \epsilon^{|\alpha|} \| D^\alpha f\|_{L^2}
\end{align*}
To sum up we get
\begin{theorem}
  If $D^\alpha f\in L^2(G)$ for all $|\alpha|\leq n$, then
  \begin{equation*}
    \| \osc_\epsilon(f) \|_{L^2}
    \leq C_\epsilon \sum_{1\leq |\alpha|\leq n} \| D^\alpha f\|_{L^2}
  \end{equation*}
  where $C_\epsilon \to 0$ as $\epsilon \to 0$.
\end{theorem}

We will need the following fact later when we obtain a Bernstein type
inequality for band-limited functions on a commutative space.  If
$\langle X,Y\rangle$ defines an inner product on $\mathfrak{g}$ and
$X_1,\dots,X_n$ is an orthonormal basis, then the associated Laplace
operator has the form $\Delta_G = D(X_1)^2 + \cdots + D(X_n)^2$. We
have:
\begin{lemma} Let the notation be as above. Then
  \begin{equation*}
    \sum_{1\leq |\alpha|\leq n} \| D^\alpha f\|_{L^2}
    \leq C \| (I-\Delta_G)^{n/2} f\|_{L^2}.
  \end{equation*}
\end{lemma}
\begin{proof}
  According to Theorem 4 in \cite{Triebel1986}
  the Sobolev norm on the left can be estimated by the Bessel norm,
  defined in \cite{Strichartz1983}, on the right.
\end{proof}

\section{Gelfand Pairs and Commutative Spaces}\label{S:4}
\noindent
In this section we introduce the basic notation for \textit{Gelfand
  pairs} and \textit{commutative spaces}. Our standard references are
\cite{D79}, Chapter 22, \cite{Dijk2009}, \cite{He84}, Chapter IV, and
\cite{Wolf}.   We give
several examples in Section \ref{se-Ex}.

Let $G$ be a connected Lie group and $K$ a compact subgroup.

\begin{theorem}\label{th-Gp} Suppose that $G$ is a connected Lie group
  and $K$ a compact subgroup. Then the following are equivalent
  \begin{enumerate}
  \item The Banach $*$-algebra $L^1(\bX)$ is commutative.
  \item The algebra $C^\infty_c(\bX)^K$ is commutative.
  \item The algebra $\bbD (\bX)$ is commutative.
  \end{enumerate}
\end{theorem}

\begin{definition} $(G,K)$ is called a Gelfand pair if one, and hence
  all, of the conditions in Theorem \ref{th-Gp} holds. In that case
  $\bX$ is called a \emph{commutative space}.
\end{definition}

If $A$ is abelian, then $(A,\{e\})$ is a Gelfand pair. Similarly, if
$K$ is a compact group that acts on the abelian group $A$ by group
homomorphisms, i.e., $a\cdot (xy)=(a\cdot x)(a\cdot y)$ then
$(G\rtimes K,K)$ is a Gelfand pair.  One of the standard ways to
decide if a given space is commutative is the following lemma:
\begin{lemma} Assume there exists a continuous involution $\tau : G\to
  G$ such that $\tau (x)\in Kx^{-1}K$ for all $x\in G$. Then $\bX=G/K$
  is commutative.
\end{lemma}
\begin{proof} As $x\mapsto x^{-1}$ is an antihomomorphism it follows
  that $f\mapsto f^\vee$ is an antihomomorphism on $L^1(\bX)^K$. On
  the other hand if we define $f^\tau (x):= f(\tau (x))$ then
  $f\mapsto f^\tau$ is a homomorphism. But as $\tau (x)=k_1x^{-1}k_2$
  it follows that $f^\vee =f^\tau$ for all $f\in L^1(\bX)^K$ and hence
  $L^1(\bX) ^K$ is abelian.
\end{proof}

\begin{example}\label{e:sphere} Let $G=\SO (d+1)$ and $K=\SO (d)$ is the group of rotations around the $e_1$-axis
\[K=\left\{\left.\begin{pmatrix} 1 & 0 \\ 0 & A\end{pmatrix}\, \right|\, A\in \SO (d)\right\}.
\]
Then $K=\{k\in G\mid k(e_1)=e_1\}$. For $a\in
G$ write $a=[a_1,\ldots ,a_{d+1}]$ where $a_j$ are the row vectors in
the matrix $a$. Then $a\cdot e_1=a_1$. If $x\in S^d$ set $a_1=x$ and
extend $a_1$ to a positively oriented orthonormal basis $a_1,\ldots
,a_{d+1}$ and set $a=[a_1,\ldots ,a_{d+1}]\in G$. Then $a\cdot
e_1=x$. This also shows that the stabilizer of $e_1$ is the group
\[K=\left\{\left. \begin{pmatrix} 1 & 0\\ 0 &
      k\end{pmatrix}\,\right|\, k\in \SO (d)\right\}\simeq \SO
(d)\, .\] Hence $S^d=G/K$. Let
\[A:=\left\{\left.a_t=\begin{pmatrix} \cos (t) & -\sin (t) & 0\\
\sin (t) & \cos (t) & 0
\\
0 & 0& I_{d-1}\end{pmatrix}\, \right|\, t\in \R\right\}\, .\]
Then every element $g\in G$ can be written as $k_1ak_2$ with $k_1,k_2\in K$ and $a\in A$. Define
\[\tau (a)=\begin{pmatrix} -1 & 0\\ 0 & I_d\end{pmatrix} a\begin{pmatrix} -1 & 0\\ 0 & I_d\end{pmatrix}\, .\]
Then $\tau|_K =\id$ and $\tau (a)=a^{-1}$ if $a\in A$. Hence $\tau (x)\in Kx^{-1}K$ which implies that $S^d$ is a commutative spaces.

Instead of working with the group it is better to work directly with the sphere. Think of $S^{d-1}$ as a subset of $S^d$ by
$v\mapsto \left(0\\ v\right)$. If $u\in S^d$ then there is a $t$ and $v\in S^{d-1}$ such that
\[u=\cos (t)e_1+\sin (t)v=k_v a_te_1\]
where $k_v$ is a rotation in $K$. The involution $\tau$ is now simply
\[u\mapsto \cos (t)e_1-\sin (t)v=k_v a_t^{-1}e_1\]
which can be rotated, using an element from $K$, back to $u$.
\end{example}

{}From now on $(G,K)$ will always--if nothings else is stated--denote
a Gelfand pair and $\bX$ will stand for a commutative space. We start
with the simple Lemma, see \cite{Dijk2009}, p. 75:
\begin{lemma} Assume that $(G,K)$ is a Gelfand pair. Then $G$ is
  unimodular.
\end{lemma}

Recall that a function $\varphi : G\to \C$ is \textit{positive
  definite} if $\varphi$ is continuous and for all $N\in \N$, all
$c_j\in \C$, and all $x_j\in G$, $j=1,\ldots ,N$, we have
\[\sum_{i,j=1}^N c_i\overline{c_j}\varphi (x_i^{-1}x_j)\ge 0\, . \]
The following gives different characterizations of positive spherical
functions. In particular, they arise as the $*$-homomorphisms of the
commutative $B^*$-algebra $L^1(\bX)^K$ and as positive definite
normalized eigenfunctions of $\bbD (\bX)$. Recall that we are always
assuming that $G$ and hence also $\bX$ is connected.
\begin{theorem}\label{th-sfct} Let $\varphi\in L^\infty (\bX)$. Then
  the following assertions are equivalent:
  \begin{enumerate}
  \item $\varphi$ is $K$-bi-invariant and $L^1(\bX)^K\to \C$,
    $f\mapsto \widehat{f}(\varphi ):= \int_G f (x)\overline{\varphi
      (x)}\, d\mu_G (x)$, is a homomorphism.
  \item $\varphi$ is continuous and for all $x,y\in G$ we have
    \[\int_G \varphi (xky)\, d\mu_K (k)=\varphi (x)\varphi (y)\, .\]
  \item $\varphi$ is $K$-bi-invariant, analytic, $\varphi (e)=1$ and
    there exists a homomorphism $\chi_\varphi : \bbD (\bX)\to \C$ such
    that
\begin{equation}\label{e:Eigenf}
D\varphi =\chi_\varphi (D) \varphi
\end{equation}
for all $D\in\bbD( X)$.
  \end{enumerate}
  The homomorphism in (a) is a $*$-homomorphism, if and only if
  $\varphi$ is positive definite.
\end{theorem}

\begin{remark} We note that (\ref{e:Eigenf}) implies that $\varphi$ is analytic because $\Delta \varphi = \chi_\varphi (\Delta
  )\varphi$ and  $\Delta$ is elliptic.
\end{remark}
\begin{definition} $\varphi \in L^\infty (\bX)^K$ is called a
  spherical function if it satisfies the conditions in Theorem
  \ref{th-sfct}.
\end{definition}

Denote by $\Psp$ the space of positive definite spherical
functions. It is a locally compact Hausdorff topological vector space
in the topology of uniform convergence on compact sets. The
\textit{spherical Fourier transform} $\cS : L^1 (\bX)^K\to \cC (\Psp)$
is the map
\[\cS (f)(\varphi )=\widehat{f}(\varphi ):= \int_G
f(x)\overline{\varphi (x)}\, d\mu_G (x)= \int_G f(x)\varphi (x^{-1})\,
d\mu_G (x)\, .\] The last equality follows from the fact that
$\overline{\varphi (x)} = \varphi (x^{-1})$ if $\varphi $ is positive
definite.  We note that $\widehat{f*g}= \widehat{f}\widehat{g}$.

\begin{theorem}\label{th-PlancherelSpherical1} There exists a unique
  measure $\mP$ on $\Psp$ such that the following holds:
  \begin{enumerate}
  \item If $f\in L^1(\bX)^K\cap L^2 (\bX)$ then
    $\|f\|_2=\|\widehat{f}\|_2$.
  \item The spherical Fourier transform extends to a unitary
    isomorphism
    \[L^2(\bX)^K \to L^2(\Psp, d\mP)\] with inverse
    \begin{equation}\label{inversion} f(\, \cdot \, ) =\int_{\Psp}
      \widehat{f}(\varphi
      )\varphi (\, \cdot \, ) \, d\mP (\varphi)
    \end{equation}
    where the integral is understood in $L^2$-sense.
  \item If $f\in L^1(\bX)^K\cap L^2(\bX)$ and $\widehat{f}\in L^1
    (\Psp,d\mP )$ then (\ref{inversion}) holds pointwise.
  \end{enumerate}
\end{theorem}

At this point we have not said much about the set $\Psp$. However, it
was proved in \cite{Ruffino} that $\Psp$ can always been identified
with a subset of $\C^s$ for some $s\in\N$ in a very simple way.
\begin{lemma} The algebra $\D (\bX)$ is finitely generated.
\end{lemma}
\begin{proof} This is the Corollary on p. 269 in \cite{He63}.
\end{proof} Let $D_1,\ldots ,D_s$ be a set of generators and define a
map
\[\Phi : \Psp \to \C^s\, ,\quad \varphi \mapsto (D_1\varphi (e),\ldots
,D_s\varphi (e))\, .\] Let $\Lambda_1 := \Phi (\Psp)$ with the
topology induced from $\C^s$, $\Lambda := \Phi (\supp \mP)$ and
$\mPh:= \Phi^* (\mP)$.
\begin{theorem}[Ruffino,\cite{Ruffino}]\label{Ruffino} The map $\Phi :
  \Psp \to \Lambda$ is a topological isomorphism.
\end{theorem}

\begin{remark} In \cite{Ruffino} the statement is for the set of
  bounded spherical functions. But $\Psp$ is a closed subset of the
  set of bounded spherical functions, so the statement holds for
  $\Psp$. Furthermore, we can choose the generators $D_j$ such that
  $\overline{D_j}=D_j$, ie., $D_j$ has real coefficients.  If
  $\varphi\in \Psp$, then $\overline{\varphi}\in \Psp$ and it follows
  that $\overline{\Lambda_1}=\Lambda_1$. We will always assume that
  this is the case.
\end{remark}

For $\lambda\in \Lambda_1$ we let
$\varphi_\lambda:=\Phi^{-1}(\lambda)$. We view the spherical Fourier
transform of $f\in L^1(\bX)^K\cap L^2(\bX)^K$ as a function on
$\Lambda$ given by $\widehat{f}(\lambda ):=
\widehat{f}(\varphi_\lambda )$.

\section{Spherical Functions and Representations}\label{S:6}
\noindent To extend Theorem \ref{th-PlancherelSpherical1} to all of
$L^2(\bX)$ one needs to connect the theory of spherical functions to
representation theory. In this section $(G,K)$ will always denote a
Gelfand pair.  A unitary representation $(\pi, \cH)$ of $G$ is called
\textit{spherical} if the space of $K$-fixed vectors
\[\cH^K:= \{u\in \cH\mid (\forall k\in K)\,\, \pi (k)u = u\}\] is
nonzero. If $(G,K)$ is a Gelfand pair then $\dim \cH^K\le 1$ for all
irreducible unitary representations of $G$.

\begin{lemma} Let $f\in L^1(\bX)^K$. Then $\pi (f)\cH \subseteq
  \cH^K$.
\end{lemma}
\begin{proof} We have for $k\in K$:
  \[\pi (k)\pi (f)v=\int_G f(x)\pi (kx)v\, dx=
  \int_G f(k^{-1}x)\pi (x)v\, dx=\int_G f(x)\pi (x)v\, dx= \pi (f)v\,
  . \qedhere\]
\end{proof}

For the following statement,
see for example Proposition 6.3.1 in \cite{Dijk2009}.
\begin{lemma} If $\cH$ is an irreducible unitary representation of $G$
  then $\dim \cH^K\le 1$.
\end{lemma}

\begin{corollary} Let $(\pi, \cH)$ be a irreducible unitary
  representation of $G$ such that $\cH^K\not=\{0\}$. Then there exists a
  $*$-homomorphism $\chi_\pi : L^1(\bX)^K\to \C$ such that
  \[\pi (f)u= \chi_\pi (f)u\] for all $u\in \cH^K$.
\end{corollary}
\begin{proof} Let $e_\pi \in \cH^K$ be an unit vector. As $\dim
  \cH^K=1$ it follows that $\cH^K=\C e_\pi$. It follows from Lemma
  \ref{le-pifHK} that $\pi (f)e_\pi =(\pi (f)e_\pi, e_\pi )e_\pi$. The
  lemma follows now by defining $\chi_\pi (f):=(\pi
  (f)e_\pi,e_\pi)$.\end{proof}

Using the heat-kernel one can show that $\cH^K\subset \cH^\omega$ but
the following is enough for us.

\begin{theorem} $\cH^K\subseteq \cH^\infty$.
\end{theorem}
\begin{proof} It is enough to show that $e_\pi \in \cH^\infty$. For
  that let $f\in C_c^\infty (\bX)$ be so that $(\pi (f)e_\pi, e_\pi )
  \not= 0$. This is possible by Lemma \ref{le-pifHK}. Let
  \[h(x)=\int_K f(kx)\, d\mu_K(k)\, .\] Then
  \[\chi_\pi (f)= (\pi (h)e_\pi , e_\pi )=(\pi (f)e_\pi ,e_\pi )\not=
  0\, .\] Hence
  \[ e_\pi =\frac{1}{\chi_\pi (f)} \pi (h)e_\pi \in \cH^K\cap
  \cH^\infty \, .\qedhere \]
\end{proof}
\begin{theorem} Let $(\pi ,\cH)$ be an irreducible spherical
  representation of $G$ and $e_\pi \in \cH^K$ a unit vector. Then the
  function
  \[\varphi_\pi (x):= (e_\pi ,\pi_\pi (x)e_\pi )\] is a positive
  definite spherical function. If $\varphi $ is a positive definite
  spherical function on $G$, then there exists an irreducible unitary
  representation $(\pi ,\cH)$ of $G$ such that $\dim \cH^K=1$ and
  $\varphi =\varphi_\pi$.
\end{theorem}
\begin{proof} Here are the main ideas of the proof. First we note that
  \[\int_K \pi (ky)e_\pi \, d\mu_K (k) =(\pi (y)e_\pi, e_\pi )e_\pi\,
  .\] Hence
  \begin{eqnarray*} \int_K \varphi_\pi (xky)\, d\mu_K (k) &=& (\pi
    (x^{-1})e_\pi ,\int_K \pi (ky)e_\pi )\, d\mu_K(k)\\ &=& (\pi
    (x^{-1})e_\pi , (\pi (y)e_\pi ,e_\pi )e_\pi)\\ &=& (e_\pi, \pi
    (x)e_\pi) (e_\pi, \pi (y)e_\pi)\\ &=&\varphi_\pi (x)\varphi_\pi
    (y)\, .
  \end{eqnarray*} Hence $\varphi_\pi$ is a spherical function. It is
  positive definite because
  \[\sum_{i,j=1}^N c_i\overline{c_j}\varphi_\pi
  (x_i^{-1}x_j)=\|\sum_{i=1}^N c_i \pi (x_i)e_\pi \|^2\ge 0\, .\qedhere\]
\end{proof}
\begin{theorem} Let $\varphi : G\to \C$ be a positive definite
  function. Then $\varphi\in \Psp$ if and only if there exists a
  irreducible spherical unitary representation $(\pi,\cH)$ and $e_\pi
  \in \cH^K$, $\|e_\pi\|=1$ such that
  \[\varphi (g)=(e_\pi ,\pi (g)e_\pi)\, .\]
\end{theorem}
\begin{proof}
  We have already seen one direction. The other direction
  follows by the classical Gelfand-Naimark-Segal construction. Assume
  that $\varphi $ is a positive definite function. Let $\cH$ denote
  the space of functions generated by linear combinations of
  $\ell(x)\varphi $, $x\in G$. Define
  \[\Big(
  \sum_{j=0}^N c_j\ell({x_j})\varphi
  ,\sum_{j=0}^Nd_j\ell({y_j})\varphi \Big)_0
  := \sum_{i,j}
  c_i\overline{d_j}\varphi (x_i^{-1}y_j)
  \]
  (By adding zeros we can
  always assume that the sum is taking over the same set of indices.)
  Then $(\,\, ,\,\,)_0$ is a positive semidefinite Hermitian form on
  $\cH_0$. Let $\cN:=\{\psi \in \cH_0\mid \|\psi \|_0=0\}$. Then $\cN$
  is $G$-invariant under left-translations and $G$ acts on $\cH_0/\cN$
  by left-translation. $(\, \, ,\,\, )_0$ defines an inner product on
  $\cH/\cN$ by $(f+\cN, g+\cN):= (f,g)$. Let $\cH$ be the completion
  of $\cH_0/\cN$ with respect with the metric given by $(\,\,
  ,\,\,)$. Then $\cH$ is a Hilbert space and the left translation on
  $\cH_0$ induces a unitary representation $\pi_\varphi$ on $\cH$. If
  $e$ is the equivalence class of $\varphi \in \cH$, then as $\varphi$
  is $K$-invariant and $\|\varphi \|_0=1$, $e\in \cH^K\setminus \{0\}$
  and $(e,\pi_\varphi (x)e)=\varphi (x)$.
\end{proof}

For $\lambda\in\Lambda_1$ and $\varphi=\varphi_\lambda$ we denote the
corresponding representation by $(\pi_\lambda ,\cH_\lambda)$. We fix
once and for all a unit vector $e_\lambda\in \cH_\lambda^K$. Let
$\pr_\lambda =\int_K \pi_\lambda (k)\, dk$. Then $\pr_\lambda$ is the
orthogonal projection $\cH_\lambda \to \cH_\lambda^K$, $\pr_\lambda
(u)=(u,e_\lambda )e_\lambda$. Let $f\in L^1(\bX)$. Then, as $f$ is
right $K$-invariant, we get $\pi_\lambda (f)=\pi_\lambda (f)\circ
\pr_\lambda$. It therefore make sense to define a vector valued
Fourier transform by
\[\widetilde {f}(\lambda ):=\pi_\lambda (f)e_\lambda\] see
\cite{os2}. We note that if $f$ is $K$-invariant, then
$\widehat{f}(\lambda )=(\widetilde{f}(\lambda ),e_\lambda)$.

If $f\in L^1(\bX)\cap L^2(\bX)$ then
\[\widetilde {\ell(x) f}=\pi_\lambda (x)\widetilde {f}(\lambda)\quad \text{ and }\quad \Tr (\pi_\lambda (f))=(\pi_\lambda (f)e_\lambda ,e_\lambda)\, . \]
Let $g=\int f^**f (kx)\, d\mu_K$. Then $g$ is $K$-binvariant and
$\widetilde{g}(\lambda )=(\pi_\lambda (f^*)\pi_\lambda (f)e_\pi, e_\pi
)=\|\widetilde{f}(\lambda )\|^2$ which is integrable on
$\Lambda$. Finally
\[\|f\|^2= f^**f(e)= g(e)= \int_\Lambda
  \widehat{g}(\lambda) \, d\mPh (\lambda )= \int_\Lambda
  \|\widetilde{f}(\lambda )\|^2 \, d\mPh (\lambda )\, .
\]
Finally, if $\lambda \mapsto (\widetilde{f}(\lambda
),e_\lambda)$ is integrable, then by the same argument as above
\[ f(x)=\ell({x^{-1}})f(e) =(\pi_\lambda (x^{-1})\widetilde{f}(\lambda
),e_\lambda)\, d\mPh (\lambda ) =\int_\Lambda (\widehat{f}(\lambda ),
\pi_\lambda (x)e_\lambda )\, d\mPh (\lambda )\, .\] Thus we have
proved the following theorem:
\begin{theorem} The vector valued Fourier transform defines a unitary
  $G$-isomorphism
  \[L^2(\bX )\simeq \int^\oplus (\pi_\lambda ,\cH_\lambda )\, d\mPh\,
  .\] If $f\in L^2(\bX)$ is so that $\lambda \mapsto
  \|\widehat{f}(\lambda )\|$ is integrable, then
  \[f(x)=\int_\Lambda (\widehat{f}(\lambda ),\pi_\lambda (x)e_\lambda
  )\, d\mPh\, .\]
\end{theorem}

\section{The Space of Bandlimited Functions}\label{S:7}
\noindent
As before $(G,K)$ denotes a Gelfand pair with $G$ connected and
$\bX=G/K$ the corresponding commutative space. In this section we
introduce the space of bandlimited functions and prove a sampling
theorem for the spaces $L^2_\Omega (\bX)$ of $\Omega$-bandlimited
functions on $\bX$.

\begin{definition} Suppose $\Omega \subset \Lambda$ be compact.  We
  say that $f\in L^2(\bX)$ is $\Omega$-\textit{bandlimited} if $\supp
  \widetilde{f}\subseteq \Omega$. $f$ is bandlimited if there exists
  $\Omega\subseteq \Lambda$ compact such that $f$ is
  $\Omega$-bandlimited.
\end{definition}

We denote by $L^2_\Omega (\bX)$ the space of $\Omega$-bandlimited
functions. As $\Omega $ will be fixed, we just say that $f$ is
bandlimited if $f\in L^2_\Omega (\bX)$. Let $\phi =\varphi_\Omega$ be
such that $\widetilde{\varphi} (\lambda) =\I_\Omega e_\lambda$. As
$\Omega$ is compact it follows that $\phi\in L^2_\Omega
(\bX)$. However, $\I_\Omega$ is in general not integrable as $\lambda
\mapsto \I_\Omega (\lambda )e_\lambda$ is not necessarily continuous.
\begin{lemma}\label{le-B1} We have
  \[\phi_\Omega (x) =\int_\Omega \varphi_\lambda (x)\, d\mPh (\lambda
  )\] is $K$-invariant and positive definite. In particular,
  $\phi_\Omega^*=\phi_\Omega $.
\end{lemma}
\begin{proof}
  The function $\lambda\mapsto (\I_\Omega (\lambda )e_\lambda
  ,\pi_\lambda (x)e_\lambda)_\lambda$ is bounded by $|(e_\lambda
  ,\pi_\lambda(x)e_\lambda )_\lambda | \I_\Omega (\lambda)\le
  \I_\Omega (\lambda)$ and hence integrable.  Therefore Theorem
  \ref{th-PlancherelSpherical1} implies that
  \[\phi_\Omega (x )=\int_\Omega (e_\lambda ,\pi_\lambda (x)e_\lambda
  )_\lambda \, d\mPh (\lambda ) =\int_\Omega \varphi_\lambda (x)\,
  d\mPh (\lambda )\, .\]

  We have
  \[
  \sum_{i,j} c_i\overline{c_j}\phi (x_i^{-1}x_j)=\int_\Omega
  \sum_{i,j}c_i\overline{c_j}\varphi_\lambda (x_i^{-1}x_j)\, d\mPh
  (\lambda )\ge 0\] as the spherical functions $\varphi_\lambda$,
  $\lambda\in \Omega$, are positive definite.
\end{proof}

\begin{theorem} $L^2_\Omega (\bX)$ is a reproducing kernel Hilbert
  space with reproducing kernel $K(x,y)=\phi_\Omega
  (y^{-1}x)$. Furthermore, the orthogonal projection $L^2(G)\to
  L^2_\Omega (\bX)$ is given by $f\mapsto f*\phi_\Omega$.
\end{theorem}
\begin{proof} We have for $f\in L^1_\Omega (\bX)$
  \[|\int_\Lambda (\widetilde {f}(\lambda ),
    \pi_\lambda (x)e_\lambda )_\lambda \, d\mPh (\lambda )|\le
    \int_\Omega \|\widetilde{f}(\lambda )\|_\lambda\, d\mPh (\lambda
    )\le |\Omega |^{1/2}\|\widetilde{f}\|^2
  \]
where $|\Omega|$ denotes the volume $\int_\Omega \,
d\mPh$ of $\Omega$ which is finite as $\Omega$ is compact.  It follows
  that
  \begin{eqnarray*}\label{eq-fx} f(x)&=&\int_\Omega
    (\widetilde{f}(\lambda ), \pi_\lambda (x)e_\lambda)_\lambda \, d\mPh
    (\lambda )\\ &=& \int_\Lambda (\widetilde{f}(\lambda ),\pi_\lambda
    (x)\I_\Omega (\lambda)e_\lambda )_\lambda\, d\mPh (\lambda)\\ &=&
    \int_\bX f(y)\overline{ \ell(x) \phi_\Omega (y) }\, dy\\ &=&\int_\bX
    f (y)\phi_\Omega (y^{-1}x)\, dy\\ &=& f*\varphi_\Omega (x)\, .
  \end{eqnarray*}
  Thus $L^2_\Omega (\bX)$ is a reproducing kernel Hilbert space with
  reproducing kernel $K(x,y)=\phi_\Omega (y^{-1}x)$. The rest follows
  now from Proposition \ref{pr-repKerHsp}.
\end{proof}

Let us point out the following consequence of Proposition
\ref{pr-repKerHsp}:
\begin{corollary} Let $f\in L^2(G)$. Then $f\in L^2_\Omega (\bX)$ if
  and only if $f*\phi_\Omega =f$.
\end{corollary}

\section{The Bernstein Inequality and sampling of bandlimited
  functions}\label{S:8}
\noindent
The definition of the topology on $\Lambda$ inspired by \cite{Ruffino}
ensures that the eigenvalues $c_\lambda$ for the Laplacian on
$p_\lambda$ are bounded when $\lambda$ is in a compact set $\Omega$.
This enables us to obtain
\begin{lemma}
  For a compact set $\Omega\in \Lambda$ the functions in $L^2_\Omega$
  are smooth and there is a constant $c(\Omega)$ such that the
  following Bernstein inequality holds
  \begin{equation*}
    \| \Delta^k f \|_{L^2} \leq c(\Omega)^k \| f\|_{L^2}
  \end{equation*}
\end{lemma}

\begin{proof}
  As we have seen, each $f\in L^2_\Omega$ can be written
  \begin{equation*}
    f(x) = \int_\Omega
    \ip{\widehat{f}(\lambda)}{\pi_\lambda(x)e_\lambda}_\lambda\,d\mPh (\lambda)
  \end{equation*}
  For fixed $\lambda$ the function
  \[t\mapsto (\widetilde{f}(\lambda ),\pi_\lambda (xe^{tX_i})e_\lambda
  )_\lambda\] is differentiable as $e_\lambda\in
  \cH_\lambda^\infty$. Thus, there exists a $t_\lambda$ between zero
  and $t$ such that
  \begin{eqnarray*}
    \frac{(\widetilde{f}(\lambda ),\pi_\lambda (xe^{tX_i})e_\lambda )_\lambda-(\widetilde{f}(\lambda ),\pi_\lambda (x)e_\lambda )_\lambda}{t}
    & = &\left( \widetilde{f}(\lambda ),\frac{\pi_\lambda (xe^{tX_i})e_\lambda -\pi_\lambda (x)e_\lambda}{t}\right)_\lambda\\
    & =&(\widetilde{f}(\lambda ), \pi_\lambda (x e^{t_\lambda X_i})\pi_\lambda (X_i)e_\lambda )_\lambda\, .
  \end{eqnarray*}
  Thus
  \begin{align*}
    \frac{f(xe^{tX_i}) - f(x)}{t} &= \int_\Omega \left(
      \widetilde{f}(\lambda), \frac{\pi_\lambda(xe^{tX_i})e_\lambda-
        \pi_\lambda(x)e_\lambda}{t}\right)_\lambda
    \,d\mPh (\lambda) \\
    &= \int_\Omega \ip{\widetilde {f}(\lambda)}{
      \pi_\lambda(x)\pi_\lambda(e^{t_\lambda
        X_i})\pi_\lambda(X_i)e_\lambda}_\lambda
    \,d\mPh (\lambda) \\
    &\leq \int_\Omega \| \widetilde {f}\|_\lambda \|
    \pi_\lambda(x)\pi_\lambda(e^{t_\lambda
      X_i})\pi_\lambda(X_i)e_\lambda\|_\lambda
    \,d\mPh (\lambda) \\
    &\leq \int_\Omega \| \widetilde {f}\|_\lambda \|
    \pi_\lambda(X_i)e_\lambda\|_\lambda
    \,d\mPh (\lambda) \\
  \end{align*}
  Here we have used that $e_\lambda$ is a smooth vector for
  $\pi_\lambda$ and the unitarity of $\pi_\lambda$. Now
  \begin{equation*}
    \| \pi_\lambda(X_i)e_\lambda\|_\lambda^2
    \leq |\ip{e_\lambda}{\sum_i \pi_\lambda(X_i)\pi_\lambda(X_i)e_\lambda }_\lambda|
    = c_\lambda \| e_\lambda\|_\lambda^2.
  \end{equation*}
  Therefore the Lebesgue dominated convergence theorem ensures that
  \begin{equation*}
    \lim_{t\to 0}\frac{f(xe^{tX_i}) - f(x)}{t}
    = \int_\Omega
    \ip{\widetilde {f}(\lambda)}{\pi_\lambda(x)\pi_\lambda(X_i)e_\lambda}_\lambda
    \,d\mPh (\lambda)
  \end{equation*}
  which shows that $f$ is differentiable.  Repeat the argument to show
  that $f$ is smooth and notice that
  \begin{equation*}
    \Delta^k f(x)
    = \int_\Omega c_\lambda^k
    \ip{\widetilde {f}(\lambda)}{\pi_\lambda(x)e_\lambda}_\lambda\,d\mPh (\lambda)
  \end{equation*}
  It then finally follows that
  \begin{equation*}
    \| \Delta^k f(x) \|_{L^2}^2
    = \int_\Omega |c_\lambda|^{2k} \| \widetilde {f}\|^2_w \,d\mu(\lambda)
    \leq c(\Omega)^{2k} \int_\Omega \| \widetilde {f}\|^2_\lambda \,d\mPh (\lambda)
  \end{equation*}
  We have thus proved the Bernstein inequality.
\end{proof}

\begin{corollary}
  Let $\Omega\subseteq \Lambda$ be a compact set and define the
  neighborhoods $U_\epsilon$ by
  $$
  U_\epsilon = \{ \exp(t_1X_1)\dots\exp(t_nX_n) \mid
  (t_1,\dots,t_n)\in [-\epsilon,\epsilon]^n \}.
  $$
  It is possible to choose $\epsilon$ small enough that for any
  $U_\epsilon$-relatively separated family $x_i$ the functions
  $\ell({x_i})\phi$ form a frame for $L^2_\Omega$.
\end{corollary}

\begin{corollary}
  Let $\Omega\subseteq \Lambda$ be a compact set and define the
  neighborhoods $U_\epsilon$ by
  $$
  U_\epsilon = \{ \exp(t_1X_1)\dots\exp(t_nX_n) \mid
  (t_1,\dots,t_n)\in [-\epsilon,\epsilon]^n \}.
  $$
  It is possible to choose $\epsilon$ small enough that for any
  $U_\epsilon$-relatively separated family $x_i$ and any partition of
  unity $0\leq \psi_i \leq 1_{x_iU_\epsilon}$ the operator
  \begin{equation*}
    T_1 f = \sum_i f(x_i)\psi_i*\phi
  \end{equation*}
  is invertible on $L^2_\Omega$.  If the functions $\ell({x_i})\phi$
  also form a frame for $L^2_\Omega$, the functions
  $T_1^{-1}(\psi_i*\phi)$ provide a dual frame.
\end{corollary}

\section{Examples of Commutative Spaces}\label{se-Ex}
\noindent
In this section we give some examples of the theory developed in the previous section. We do not discuss the Riemannian symmetric spaces of the compact type as those can be found in \cite{Pesenson2008}.

\subsection{The Space $\R^d$} The simplest example of a Gelfand pairs
is $(\R^d, \{0\})$. The algebra of invariant differential operators is
$\bbD (\R^d)=\C[\partial_1,\ldots ,\partial_d]$ the polynomials in the
partial derivatives $\partial_j=\partial/\partial x_j$. The positive
definite spherical functions are the exponentials $\varphi_{\lambda} (x)
= e^{i\lambda \cdot x}$, $\lambda\in\R^d$. Using $\partial_1,\ldots
,\partial_d$ as generators for $\bbD(\R^d)$ Theorem \ref{Ruffino}
identifies $\Lambda$ with $i\R^d$ via the map $\varphi_{\lambda} \mapsto
i(\lambda_1,\ldots ,\lambda_d)$. Note the slight different from our previous notation for
$\varphi_\lambda$.

We can also consider $\R^d$ as the commutative space corresponding to
the connected Euclidean motion group $G=\SO (d)\ltimes \R^d$ with
$K=\SO (d)$. The $K$-invariant functions are now the radial
functions $f(x)=F_f(\|x\|)$, where $F_f$ is a function of one
variable. We have $\bbD (\R^d)=\C[-\Delta]$ and Theorem \ref{Ruffino}
now identifies the spectrum $\Lambda$ with $\R^+$. For $\lambda\in \R$
we denote by $\varphi_\lambda$ the spherical function with $-\Delta
\varphi_\lambda =\lambda^2\varphi_\lambda$.

Denote by $J_\nu$ the Bessel-function
\[J_\nu (r)=\frac{(r/2)^{\nu}}{\Gamma (1/2)\Gamma (\nu
  +1/2)}\int_{-1}^1\cos (tr) (1-t^2)^{\nu -1/2}\, dt\] see
\cite{Lebedev1972}, p. 144.
\begin{lemma} $\displaystyle \varphi_\lambda (x)=
  \frac{2^{\frac{d-2}{2}}\Gamma\left(\frac{d}{2}\right)}{(\lambda
    \|x\|)^{\frac{d-2}{2}}}\, J_{(d-2)/2} (\lambda \|x\| )=
  \frac{\Gamma \left(\frac{d}{2}\right)}{\sqrt{\pi}\Gamma
    \left(\frac{d-1}{2}\right)} \int_{-1}^1\cos ( \lambda \|x\| t )
  (1-t^2)^{\frac{d-3}{2} }\, dt $.
\end{lemma}
\begin{proof} Denote for the moment the right hand side by
  $\psi_\lambda$. Then $\psi_\lambda $ is analytic as $\cos$ is
  even. It is also an radial eigenfunction of $-\Delta$ with
  eigenvalue $\lambda^2$ and $\psi_\lambda (0)=1$. Now Theorem
  \ref{th-sfct} implies that $\varphi_\lambda=\psi_\lambda$.
\end{proof}

\begin{remark} We note that we can write
  \[\varphi_\lambda (x)=\int_{S^{d-1}}e^{-i\lambda (\omega ,x)}\,
  d\sigma (\omega)\] where $d\sigma$ is the normalized rotational
  invariant measure on the sphere.
\end{remark}

It is easy to describe the representation $(\pi_\lambda ,\cH_\lambda)$
associated to $\varphi_\lambda$.  For $\lambda\in\R^*$ set
$\cH_\lambda =L^2(S^{d-1},d\sigma )=L^2(S^{d-1})$ and define
\[ \pi_\lambda ((k,x)) u (\omega ) := e^{-i\lambda (\omega ,x)}u
(k^{-1}(\omega ))\, .\] We take the constant function $\omega \mapsto
1$ as normalized $K$-invariant vector $e_\lambda$. Then
\[(e_\lambda ,\pi_\lambda ((k,x))e_\lambda)= \int_{S^{d-1}}
\overline{e^{-i\lambda (\omega ,x)}}\, d\sigma (\omega
)=\varphi_\lambda (x)\, .\] We refer to \cite{os2} for more
information.

\subsection{The Sphere $S^d$} Let $S^d=\{x\in\R^{d+1}\mid \|x\|=1\}$
be the unit sphere in $\R^{d+1}$. We refer to Chapter 9 of
\cite{Faraut2008} and Chapter III in \cite{Takeuchi1994} for more
detailed discussion on harmonic analysis and representation theory
related to the sphere. In particular, most of the proofs can be found
there. Recall from Example \ref{e:sphere} that $S^d=G/K$ where $G=\SO (d+1)$ and
$K=\SO (d)$ and that $S^d$ is a commutative space.

For $d=1$ we have $S^1=\bbT=\{z\in \C\mid |z|=1\}$ is an abelian group
and the spherical functions are just the usual characters $z\mapsto
z^n$ (or, if we view $\bbT=\R/2\pi \Z$, the functions $\theta \mapsto
e^{in \theta}$). We therefore assume that $d\ge 2$, but we would also
like to point out another special case.  $\SO(3)\simeq \SU (2)$ and
$\SO_o(4)=\SU (2)\times \SU (2)$. The group $K=\SU (2)$ is embedded as
the diagonal in $\SU (2)\times \SU (2)$. Then
\[S^3=\SO(4)/\SO(3)\simeq \SU (2)\simeq \{z\in\bbH\mid |z|=1\}\]
and the $K$-invariant functions on $S^3$ corresponds to the central
functions on $\SU (2)$, i.e. $f(kuk^{-1})=f(u)$. Hence $\Lambda$, the
set of spherical representations, is just $\widehat{\SU (2)}$, the set
of equivalence classes of irreducible representations of $\SU (2)$ and
the spherical functions are
\[\varphi_\pi =\frac{1}{d(\pi )}\Tr \pi =\frac{1}{d (\pi)} \chi_\pi\]
where $d(\pi )$ denotes the dimension of $V_\pi$. We will come back to
this example later.

Denote by $\fg=\so (d+1)=\{X\in M_{d+1}(\R )\mid X^T=-X\}$ the Lie
algebra of $G$. We can take $\langle X, Y\rangle = -\Tr (XY)$ as a $K$-invariant inner product
on $\fg$. Then
\[\fk=\left\{\left.\begin{pmatrix} 0 & 0\\ 0 & Y\end{pmatrix}\,
  \right|\, Y\in \so (d)\right\}\simeq \so (d)\]
and $\fs=\fk^{\perp}$ is given by
\[\fs=\left\{\left. X(v)=\begin{pmatrix} 0 &-v^T\\ v &
      0\end{pmatrix}\, \right|\, v\in\R^d\right\}\simeq \R^d\, .\]

A simple matrix multiplication shows that $k X(v)k^{-1}=X(k(v))$ where
we have identified $k\in \SO(d)$ with its image in $K$. It follows
that the only invariant polynomials on $\fs$ are those of the form
$p(X(v))=q(\|v\|^2)$ where $q$ is a polynomial of one variable. It
follows that $\bbD (S^d)=\C[\Delta ]$ where $\Delta$ now denotes the
Laplace operator on $S^{d-1}$. Thus $\bbD (S^d)$ is abelian and hence
$S^d=\SO(d+1)/\SO(d)$ is a commutative space.

Recall that a polynomial $p(x)$ on $\R^{d+1}$ is homogeneous of degree
$n$ if $p(\lambda x)=\lambda^n p(x)$ for all $\lambda\in \R$. $p$ is
harmonic if $\Delta_{\R^{d+1}}p=0$. Denote by $\cH_n $ the space of
harmonic polynomials that are homogeneous of degree $n$ and set
\begin{equation}\label{defYn}
  \cY_n := \cH_n|_{S^d}=\{p|_{S^d}\mid p\in \cH_n\}\, .
\end{equation}

As the action of $G$ on $\R^{d+1}$ commutes with $\Delta_{\R^{d+1}}$
it follows that each of the spaces $\cY_n$ are $G$-invariant. Denote
the corresponding representation by $\pi_n$. Thus $\pi_n
(a)p(x)=p(a^{-1}x)$ for $p\in \cY_n$.
\begin{theorem} The following holds:
  \begin{enumerate}
  \item $(\pi_n,\cY_n)$ is an irreducible spherical representation of
    $\SO_o(d+1)$.
  \item If $(\pi,V)$ is an irreducible spherical representation of $G$
    then there exists an $n$ such that $(\pi,V)\simeq (\pi_n,\cY_n)$.
  \item $\dim \cY_n = (2n+d-1)\frac{(d+n-2)!}{(d-1)!n!}=: d(n)$.
  \item $-\Delta |_{\cY_n} = n(d+n-1)$.
  \item $L^2(S^d)\simeq_G \bigoplus_{n=0}^\infty \cY_n$. In
    particular, every $f\in L^2(S^d)$ can be approximated by harmonic
    polynomials.
  \end{enumerate}
\end{theorem}

The last part of the above theorem implies that $\Lambda
=\N=\{0,1,\ldots \}$. We use this natural parametrization of $\Lambda$ rather than the one given
in Section \ref{S:6}.

For $\Omega\in \N$
the Paley-Wiener space for $\Omega$ is
\[L^2_\Omega (S^d)=\{p|_{S^d}\mid p \text{ is a harmonic polynomial of
  degree }\le \Omega\}\, .\] It is noted that $\dim L^2_\Omega
(S^d)<\infty$ which is also the case in the more general case of
compact Gelfand pairs.

The group $\SO(d)$ acts transitively on spheres in $\R^d$. Hence
every $v\in S^d$ is $K$-conjugate to a vector of the form $(\cos
(\theta ), \sin (\theta), 0,\ldots ,0)^T$ and a function
$f$ is $K$-invariant
if and only if there exists a function $F_f$ of one variable such that
\[f(v)=F_f(\cos (\theta ))=F_f ((v,e_1))=F_f(v_1)\, .\] In particular,
this holds for the spherical function $\varphi_n (x)$ corresponding to
the representation $\pi_n$ as well as the reproducing kernel $\phi$ of
the space $L^2_\Omega (S^d)$. In fact, for $d\ge 2$ the spherical
functions are determined by
the Jacobi polynomials, or normalized
Gegenbauer polynomials in the following manner:
$F_{\phi_n}(t)=\Phi_n(t)$ or $\varphi_n(x)=\Phi_n
((x,e_1))=\Phi_n(\cos (\theta ))$,
where
\begin{eqnarray*}
  \Phi_n(\cos (\theta))&=&{}_2F_1(n+d-1,-n,\frac{d}{2};\sin^2(\theta
  /2)) \\
&=&{}_2F_1\left(n+d-1,-n,\frac{d}{2};\frac{1-\cos (\theta )}{2}\right)\\
  &=& \frac{n! (d-2)!}{(n+d-2)!} C^{(d-1)/2}_n(\cos (\theta )).
\end{eqnarray*}
As the polynomials $\varphi_n(t)$ are real valued we can write the
spherical Fourier transform as
\[\widehat{f}(n)=\int_{S^d} f(x)\varphi_n(x )\, d\sigma (x)=
\frac{\Gamma\left(\frac{d+1}{2}\right)}{\sqrt{\pi}\Gamma\left(\frac{d}{2}\right)}\,
\int_{-1}^1 F_f(t)\Phi_n(t)(1-t^2)^{\frac{d}{2}-1}\, dt\] with
inversion formula
\[f(x)=\sum_{n=0}^\infty d(n)\widehat{f}(n)\varphi_n(x)=
\sum_{n=0}^\infty d(n)\widehat{f}(n)\Phi_n((x,e_1))\, .\] In
particular the sinc-type function is given by
\begin{equation}\label{phiSphere}
  \phi_\Omega (x)=F_{\phi_\Omega} ((x,e_1))=\sum_{n=0}^\Omega d(n)\Phi_n((x,e_1))\, .
\end{equation}
Note also that we can write the convolution kernel $\phi_\Omega
(a^{-1}b)$, $a,b\in \SO(d+1)$ as $F_{\phi_\Omega} ((x,y))$ if $x=be_1$ and $y=a e_1$.

For $d=1$ the sphere is the torus $\bT=\{z\in \C\mid |z|=1\}$ and
$\varphi_n (z)=z^n$. Hence
\[\phi_{\Omega}(e^{it})=\sum_{n=-\Omega}^\Omega e^{nit}=\frac{\sin
  ((\Omega +1/2)t)}{\sin (t/2)}\] is the Dirichlet kernel $D_\Omega$.
In
the higher dimensional cases the kernel $\phi_\Omega$ behaves very
similar to
the Dirichlet kernel. Here are some of its properties:
\begin{lemma}\label{phi-prop}
  Let the notation be as above. Then the following holds:
  \begin{enumerate}
  \item $\phi_\Omega (e_1)=\sum_{n=0}^\Omega d(n)=\dim L^2_\Omega
    (S^d) \nearrow \infty$ as $\Omega \to \infty$.
  \item $\int_{S^d}\phi_\Omega (x)\, d\sigma (x)=1$.
  \item $\|\phi_\Omega \|^2_2=\sum_{n=0}^\Omega d(n)\to \infty \text{
      as } \Omega\to \infty$.
  \item If $f\in L^2(S^d)$, then $\displaystyle{f*\phi_\Omega
      =\int_{S^{d-1}}f(x)F_{\phi_\Omega} ((\cdot ,x))\,d\sigma (x)
      \lar_{\Omega\to \infty} f}$ in $L^2(S^d)$.
  \end{enumerate}
\end{lemma}
Let $N(\Omega )=\dim L^2_\Omega (S^d)=1+d(1)+\ldots +d(\Omega)$. Then
every set of points $\{\omega_j\in S^d\mid j=1,\ldots ,N(\Omega )\}$
such that the functions $F_{\phi_\Omega}((\cdot ,\omega_j))$ are linearly independent
will give us a basis (and hence a frame) for $L^2_\Omega (S^d)$. Further
$N(\Omega )$ is the minimal number of points so that the sampling will
determine a arbitrary function $f\in L^2_\Omega (S^d)$.
If $n> N(\Omega )$, then the
functions $\{F_{\phi_\Omega} ((\cdot ,\omega_j))\}_{j=1}^n$
will form a frame if
and only if the set is generating.

Let us come back to the special case $S^3\simeq \SU (2)$. The set
$\Lambda\simeq \widehat{\SU (2)}$ is isomorphic to $\N$ in such a way
that $d (n)=d (\pi_n)=n+1$. Every element in $\SU (2)$ is conjugate to
a matrix of the form
\[u(\theta )=\begin{pmatrix} e^{i\theta }&0\\ 0 &
  e^{-i\theta}\end{pmatrix}\, .\] We have
\[\varphi_n(u(\theta ))=\frac{1}{n+1}\chi_{\pi_n}(u(\theta
))=\frac{1}{n+1}\frac{\sin ((n+1)\theta)}{\sin (\theta )}\, .\]
It follows that
\begin{eqnarray*}
  \phi_\Omega (u(\theta ))& =&\frac{1}{\sin (\theta )}\sum_{n=1}^{\Omega+1}\sin (n\theta )\\
  &=&\frac{1}{2 i \sin (\theta )}\left(\sum_{n=1}^{\Omega +1}\left(e^{i\theta }\right)^n -\sum_{n=1}^{\Omega +1}\left(e^{-i\theta }\right)^n\right)\\
  &=& \frac{\sin ((\Omega +2)\theta/2)\sin ((\Omega +1)\theta/2)}{\sin (\theta )\sin (\theta /2)}\, .
\end{eqnarray*}
\subsection{Symmetric Spaces of the Compact Type} We will not discuss
the general case of symmetric spaces $\bX=G/K$ of the compact type to
avoid introducing too much new notation, but the general case follows
very much the same line as the special case of the sphere. Recall that
that ``symmetric space of the compact type'' means that the group $G$
is compact and there exists an involution $\tau : G\to G$ such that
with $G^\tau =\{u\in G\mid \tau (u)=u\}$ we have
\[ (G^\tau)_o\subseteq K\subseteq G^\tau\, .\] An example is the
sphere $S^d$ where as in the last subsection $G=\SO_o(d+1)$ and the
involution $\tau $ is given by
\[u\mapsto \begin{pmatrix} -1 & 0\\ 0 &
  I_d\end{pmatrix}u\begin{pmatrix} -1 & 0\\ 0 & I_d
\end{pmatrix}=\begin{pmatrix} u_{11} & -v^t \\ -v & k\end{pmatrix}\]
as in Example \ref{e:sphere}. All of those spaces are commutative. The
spectral set $\Lambda$ is well understood, see \cite{He84}, Theorem
4.1, p. 535. In particular $\Lambda$ is discrete. Each representation
$(\pi_\lambda ,\cH_\lambda)$ occur with multiplicity one in
$L^2(\bX)$. Denote the image by $\cY_\lambda$. Then if $\Omega$ is
given, there exists a finite set $\Lambda (\Omega )\subset \Lambda$
such that
\[ L^2_\Omega (\bX) =\bigoplus_{\lambda \in\Lambda (\Omega
  )}\cY_\lambda\] and $N (\Omega )=\dim L^2_\Omega (\bX )$ is
finite. In particular, only finitely many points are needed to
determine the elements in $L^2_\Omega (\bX)$.

The spherical functions are well understood. They are given by the
generalized hypergeometric functions (and Jacobi polynomials) of
Heckman and Opdam \cite{HeSch94}. Again
\[\phi_\Omega =\sum_{\lambda \in \Lambda (\Omega )} d(\pi_\lambda
)\varphi_\lambda \in C^\omega (\bX)\] the Dirichlet
kernel. Furthermore, Lemma \ref{phi-prop} holds true in the general
case.

\subsection{Gelfand pairs for the Heisenberg group}

For the details in the following
discussion we refer to
\cite{Strichartz1991} and \cite{Benson1992}.
We let $\bbH_n=\C^n\times\R$ denote the
$2n+1$-dimensional Heisenberg group with group composition
\begin{equation*}
  (z,t)(z',t') = (z+z',t+t'+\frac{1}{2}\mathrm{Im}(\overline{z}z')).
\end{equation*}
Denoting $z=x+iy$ the Heisenberg group is equipped with the left and
right Haar measure $dx\,dy\,dt$ where $dx,dy,dt$ are Lebesgue meaures
on $\R^n$,$\R^n$ and $\R$ respectively.  The group $K=\rU (n)$ acts on
$\bbH_n$ by group homomorphism given by
\begin{equation*}
  k\cdot (z,t) = (kz,t)\, .
\end{equation*}
Let $G=K\ltimes \bbH_n$.  It follows that $L^p(G/K) \simeq
L^p(\bbH_n)$ and $L^p(G/K)^K \simeq L^p(\bbH_n)^K =
L^p(\bbH_n)_{rad}$.  It is known \cite{Benson1990}
that
the algebra $L^1(\bbH_n)^K$ of integrable radial functions on $\bbH_n$
is commutative and thus $(G,K)$ is a Gelfand pair.  This is also the
case for several other subgroups of $U(n)$ as shown in
\cite{Benson1990}.

\subsubsection{Representation Theory for $G$}

A collection of important representations for the Heisenberg group are
the Bargman representations for $\lambda > 0$ given by
\begin{equation*}
  \pi_\lambda(z,t)f(w)
  = e^{i\lambda t -\lambda \mathrm{Im}(w\overline{z})/2 -\lambda |z|^2/4}f(w+z)
\end{equation*}
is the Bargman representation of the Fock space $\cF_\lambda$ of
entire functions on $\C^n$ for which
\begin{equation*}
  \| F\|_\lambda^2
  = \Big(\frac{\lambda}{2\pi} \Big)^n
  \int_{\C^n} |F(z)|^2e^{-\lambda |z|^2/2}\,dz < \infty.
\end{equation*}
For $-\lambda<0$ define the representations
\begin{equation*}
  \pi_{-\lambda}(z,t)f(w) = \pi_\lambda(\overline{z},t)f(w)
\end{equation*}
on the anti-holomorphic functions $\overline{\cF_\lambda}$.  These
representation are irreducible and the left regular representation of
$\bbH_n$ on $L^2(\bbH_n)$ decomposes as
\begin{equation*}
  \int^\oplus_{\R^*} (\pi_\lambda,\cF_\lambda) \,|\lambda|^n \,d\lambda.
\end{equation*}
We should note that there are more irreducible representations than
the $\pi_\lambda$, but they are one-dimensional and do not show up in
the Plancherel formula (they are of Plancherel measure 0).

Let us now turn to the regular representation of $G$ on $L^2(\bbH_n)$
given by
\begin{equation*}
  \ell(k,z,t)f(z',t')
  = f(k^{-1}(z'-z),t'-t-\frac{1}{2}\mathrm{Im}(z'\overline{z})).
\end{equation*}
Notice that $U(n)$ acts only on the $z$-variable, and for fixed $k\in
U(n)$ the elements $G_k = \{(kz,t) \mid (z,t)\in \bbH_n \}$ is a group
isomorphic to $\bbH_n$. Thus the left regular representation of $G_k$
on $L^2(G_k)$ can be decomposed using the Bargman representations. We
get
\begin{equation*}
  (\ell,L^2(G_k))
  =  \int^\oplus_{\R^*} (\pi_\lambda^k,\cF_\lambda) \,|\lambda|^n \,d\lambda
\end{equation*}
where $\pi_\lambda^k(z,t) = \pi_\lambda(kz,t)$.  Note that with
$\nu(k)f(w) = f(k^{-1}w)$ we have
$$\pi_\lambda^k(z,t) = \nu(k)\pi_\lambda(z,t)\nu(k)^{-1}.$$
Denote by $\pi_{\lambda,m}$ the representation $\nu$ restricted to the
homogeneous polynomials of degree $m$ $V_{\lambda,m}$.  Then
$(\nu,\cF_\lambda)$ decomposes into
\begin{equation*}
  \bigoplus_{m=0}^\infty (\pi_{\lambda,m},V_{\lambda,m})
\end{equation*}
Note that $\dim(V_{\lambda,m})=2m+n$, $\dim(V_{\lambda,m}^K)=1$.  Let
$H_{\lambda,m}$ be the Hilbert space spanned by
$\pi_\lambda(G )u_\lambda$ with $u_\lambda$ in
$V_{\lambda,m}^K$. The representations of $G$ on $H_{\lambda,m}$ thus
obtained are irreducible \cite{Strichartz1991},
and provide us with a
decomposition of the left regular representation of $G$ on
$L^2(\bbH_n)$:
\begin{equation*}
  (\ell,L^2(\bbH_n))
  = \bigoplus_{m=0}^\infty \int^\oplus_{\R^*}
  (\pi_{\lambda,m},H_{\lambda,m}) \,|\lambda|^n \,d\lambda.
\end{equation*}

\subsubsection{Spherical functions}
The bounded $U(n)$-spherical functions in this case are
\begin{align*}
  \phi_{\lambda,m}(z,t) = \int_{U(n)}
  (\pi_\lambda^k(z,t)u_\lambda,u_\lambda)_{\cF_\lambda}\,dk =
  e^{i\lambda t} L_m^{(n-1)}\Big(|\lambda||z|^2/2
  \Big)e^{-|\lambda||z|^2/4}
\end{align*}
for $\lambda\in\R\setminus\{ 0\}$ and $m=0,1,2,\cdots$.
Here $L_m^{(n-1)}$ is the Laguerre polonomial of degree $m$ and order
$n-1$
\begin{equation*}
  L_m^{(n-1)}(x)
  = {m+n-1  \choose m}^{-1}
  \sum_{k=0}^m (-1)^k {m+n-1  \choose m-k} \frac{x^k}{k!}
\end{equation*}

In this fashion the spectrum for $L^1(\bbH_n)^K$ can be identified
with the Heisenberg fan \cite{Strichartz1991}
\begin{equation*}
  \fF_n
  = \{ ((2m+n)|\lambda|,\lambda) \mid m\in\N_0,\lambda\neq 0\}
  \cup \R_+
\end{equation*}
with Plancherel measure supported on $\Lambda = \{
((2m+n)|\lambda|,\lambda) \mid m\in\N_0,\lambda\neq 0\}$ and given
explicitly as
\begin{equation*}
  \int_\Lambda F(\phi)\,d\mu(\phi)
  = \int_{\R^*} \sum_{m\in\N_0}
  (2m+n) F(\phi_{\lambda,m})|\lambda|^n\,d\lambda.
\end{equation*}
As shown by \cite{Benson1996} and more generally in \cite{Ruffino} the
topologies on $\fF_n$ and $\Lambda$ is the topology inherited from
$\R^2$.

\subsubsection{Sampling and oversampling of band-limited functions}
Let $L^2_\Omega(\bbH_n)$ be the space of functions in
$L^2(\bbH_n)$ with Fourier transform supported in
$$\Omega = ((2m+n)|\lambda|,\lambda) \mid m=0,\cdots,M ;0<|\lambda|\leq R\}.$$
In this case the sinc-type function is given by the integral
\begin{align*}
  \phi(z,t) &= \sum_{m=1}^M \int_{0}^R e^{i\lambda t}
  L_m^{(n-1)}\Big(|\lambda||z|^2/2 \Big)e^{-|\lambda||z|^2/4}
  \,d\lambda \\ &\qquad+ \int_{0}^R e^{-i\lambda t}
  L_m^{(n-1)}\Big(|\lambda||z|^2/2 \Big)e^{-|\lambda||z|^2/4}
  \,d\lambda \\
  &= \sum_{m=1}^M \int_{0}^R 2\cos(\lambda t)
  L_m^{(n-1)}\Big(|\lambda||z|^2/2 \Big)e^{-|\lambda||z|^2/4}
  \,d\lambda.
\end{align*}
Let $x_iU$
with $x_i\in G$ be a cover of the group $G=\bbH_n\ltimes K$, then
$x_iKUK$ covers the Heisenberg group $\bbH_n$.  Let $\psi_i$ be a
bounded partition of unity, which could for example characteristic
functions for disjoint sets $U_i\in x_iU$.  The operator $T$ then has
the form
\begin{equation*}
  Tf = \sum_i f(x_iK)\psi_i*\phi
\end{equation*}
where $\phi$ is given above. Choosing $x_i$ close enough we can invert
$T$ to obtain sampling results.

Another interesting application is related to oversampling.  Let
$\widehat{\phi}$ be a compactly supported Schwartz function on $\R^2$
such that its restriction to $\Omega$ is $1_\Omega$.  Let us say the
support is in $\Omega_1$.  According to \cite{Astengo2009} there is a
Schwartz function $\phi$ on $\bbH_n$ such that its Fourier transform
is equal to $\widehat{\phi}$ restricted to $\Omega_1$.  The function
$\phi$ is therefore both integrable and band-limited.
Let $\phi_1$ be the sinc function associated to $\Omega_1$, then
choosing $x_i$ close enough (closer than for the operator $T$ to
ensure that also $T_1$ is invertible) the operator
\begin{equation*}
  T_1 f = \sum_i f(x_iK)\psi_i*\phi_1
\end{equation*}
becomes invertible on $L^2_{\Omega_1}(\bbH_n)$. Therefore
\begin{equation*}
  f = \sum_i f(x_iK) T_1^{-1}(\psi_i*\phi_1)
\end{equation*}
with convergence in $L^2_{\Omega_1}(\bbH_n)$.  For
$f\in L^2_\Omega(\bbH_n)$ we then
also get, since $f=f*\phi$, that
\begin{equation*}
  f = \sum_i f(x_iK) T_1^{-1}(\psi_i*\phi_1)*\phi
\end{equation*}
and this time with convergence in $L^2_\Omega(\bbH_n)$.

\begin{remark}
  This oversampling situation is not possible for symmetric spaces of
  non-compact type. The reason is that there are no integrable
  band-limited functions with Fourier transform constant on a set with
  limit point.
\end{remark}

\end{document}